\documentclass[12pt, twoside]{article}
\usepackage{amsfonts,mathrsfs,amsmath,amssymb,amsthm,upref,amscd}
\usepackage{xcolor}
\usepackage[all,cmtip]{xy}
\usepackage{setspace}
\usepackage{enumerate}
\usepackage[titletoc]{appendix}
\usepackage{times}
\usepackage{cite}
\usepackage{tikz}
\usepackage[colorinlistoftodos]{todonotes}
\usetikzlibrary{arrows}
\numberwithin{equation}{section}

\def\titlerunning#1{\gdef\titrun{#1}}
\makeatletter
\def\author#1{\gdef\autrun{\def\and{\unskip, }#1}\gdef\@author{#1}}
\def\address#1{{\def\and{\\\hspace*{18pt}}\renewcommand{\thefootnote}{}%
\footnote {#1}}%
\markboth{\autrun}{\titrun}}
\makeatother
\def\email#1{e-mail: #1}
\def\subjclass#1{{\renewcommand{\thefootnote}{}%
\footnote{\emph{Mathematics Subject Classification (2010):} #1}}}
\def\keywords#1{\par\medskip
\noindent\textbf{Keywords.} #1}

\allowdisplaybreaks

\theoremstyle{plain}
\newtheorem{Thm}{Theorem}[section]
\newtheorem{Lem}[Thm]{Lemma}
\newtheorem{Cor}[Thm]{Corollary}
\newtheorem{Prop}[Thm]{Proposition}
\newtheorem*{Thm*}{Theorem}
\newtheorem*{claim*}{Claim}
\newtheorem*{Cor*}{Corollary}

\newtheorem*{Ques*}{Question}

\newtheorem*{Prob*}{Problem}
\newtheorem*{OProb*}{Open Problem}
\theoremstyle{definition}

\newtheorem*{Def*}{Definition}
\newtheorem{Rem}[Thm]{Remark}

\DeclareMathOperator{\real}{Re}

\DeclareMathOperator{\vol}{vol}

\DeclareMathOperator{\supp}{supp}

\DeclareMathOperator{\Vol}{Vol}
\DeclareMathOperator{\Spin}{Spin}
\DeclareMathOperator{\SO}{SO}
\DeclareMathOperator{\Aut}{Aut}

\newcommand{\equ}{equation}

\newcommand{\R}{\mathbb{R}}

\newcommand\ch{\mathcal{H}}

\newcommand\cj{\mathcal{J}}
\newcommand\ck{\mathcal{K}}
\newcommand\cl{\mathcal{L}}
\newcommand\cm{\mathcal{M}}

\newcommand\cs{\mathcal{S}}

\newcommand{\inp}[2]{\left\langle#1,#2\right\rangle}
\newcommand{\normm}[1]{{\left\vert\kern-0.25ex\left\vert\kern-0.25ex\left\vert #1 
		\right\vert\kern-0.25ex\right\vert\kern-0.25ex\right\vert}}

\def\mbs{\mathbb{S}}

\def\msn{\mathscr{N}}

\def\msp{\mathscr{P}}

\def\ig{\textit{g}}

\def\ov{\overline}

\def\pa {\partial}
\def\op{\oplus}

\def\De{\Delta}
\def\ka{\kappa}
\def\al{\alpha}
\def\bt{\beta}

\def\de{\delta}
\def\Ga{\Gamma}
\def\ga{\gamma}

\def\lm{\lambda}
\def\La{\Lambda}
\def\om{\omega}

\def\sa{\sigma}
\def\vr{\varepsilon}
\def\va{\varphi}

\begin{document}
	
\titlerunning{On the B\"ar-Hijazi-Lott invariant}

\title{On the B\"ar-Hijazi-Lott invariant for the Dirac operator and a spinorial proof of the Yamabe problem}

\author{Yannick Sire \quad \& \quad Tian Xu}%\footnote{Supported by the National Science Foundation of China (NSFC 11601370) and the Alexander von Humboldt Foundation of Germany}

\date{}

\maketitle

\address{
	Y. Sire: Department of Mathematics, Johns Hopkins University, 3400 N. Charles Street, Baltimore, Maryland 21218;
	\email{ysire1@jhu.edu}
	\and
	T. Xu: Department of Mathematics, Zhejiang Normal University, Jinhua, Zhejiang, 321004, China; \email{xutian@amss.ac.cn} 
}

\subjclass{Primary 53C27; Secondary 35R01}

\begin{abstract}

Let $M$ be a closed spin manifold of dimension $m\geq6$ equipped with a Riemannian metric
$\ig$ and a spin structure $\sa$.  Let $\lm_1^+(\tilde\ig)$ be the smallest positive eigenvalue of the Dirac operator $D_{\tilde\ig}$ on $M$ with respect to a metric $\tilde\ig$ conformal to $\ig$. The B\"ar-Hijazi-Lott invariant is defined by $\lm_{min}^+(M,\ig,\sa)=\inf_{\tilde\ig\in[\ig]}\lm_1^+(\tilde\ig)\Vol(M,\tilde\ig)^\frac{1}{m}$. In this paper, we show that
\[
\lm_{min}^+(M,\ig,\sa)<\lm_{min}^+(S^m,\ig_{S^m},\sa_{S^m})=\frac m2\Vol(S^m,\ig_{S^m})^{\frac1m}
\]
provided that $\ig$ is not locally conformally flat.  This estimate is a spinorial analogue to an estimate by
T. Aubin, solving the Yamabe problem in this setting.

	\vspace{.5cm}
	\keywords{Dirac operator; Spinorial Yamabe equation;  conformal invariant.}
\end{abstract}

%\tableofcontents

\section{Introduction}

A central question in the study of geometric operators is that of how much information is needed to estimate the eigenvalues of an operator. Among the most important operators are the conformal Laplacian and the Dirac operator. On an $m$-dimensional closed Riemannian manifold $(M,\ig)$, $m\geq3$, we denote the first eigenvalue of the conformal Laplacian $L_\ig:=4\frac{m-1}{m-2}\De_\ig+S_\ig$ (also called the Yamabe operator) by $\lm_1(L_\ig)$. Let $[\ig]$ be the collection of all metrics conformally equivalent to $\ig$, we define
\begin{\equ}\label{Yamabe invariant}
Y(M,\ig):=\inf_{\tilde\ig\in[\ig]}\lm_1(L_{\tilde\ig})\Vol(M,\tilde\ig)^{\frac2m}.
\end{\equ}
This quantity plays a crucial role in the analysis of the famous Yamabe problem. If $(M,\ig)$ is not conformally equivalent to the standard sphere $(S^m,\ig_{S^m})$, it is proven that
\begin{\equ}\label{Yamabe-inequ}
Y(M,\ig)<Y(S^m,\ig_{S^m})=m(m-1)\om_m^{\frac2m},
\end{\equ}
where $\om_m$ stands for the volume of the standard sphere $S^m$. This estimate was the final step in solving the Yamabe problem (cf. \cite{Aubin, Schoen, Trudinger, Yamabe} and \cite{LeeParker} for a good overview).

In the setting of spin geometry there exists a conformally covariant operator, the Dirac operator, which enjoys analogous properties to the conformal Laplacian. This operator was formally introduced by M.F. Atiyah in 1962 in connection with his elaboration of the index theory of elliptic operators.

Let $(M,\ig,\sa)$ be an $m$-dimensional closed spin manifold, $m\geq2$, with a fixed Riemannian metric $\ig$ and a fixed spin structure $\sa:P_{\Spin}(M)\to P_{\SO}(M)$. The Dirac operator $D_\ig$ is defined in terms of a representation $\rho:\Spin(m)\to\Aut(\mbs_m)$ of the spin group which is compatible with Clifford multiplication. Let $\mbs(M):=P_{\Spin}(M)\times_\rho\mbs_m$ be the associated bundle, which we call the spinor bundle over $M$. Then the Dirac operator $D_\ig$ acts on smooth sections of $\mbs(M)$, i.e. $D_\ig: C^\infty(M,\mbs(M))\to C^\infty(M,\mbs(M))$.

Let $L^2(M,\mbs(M))$ be the Hilbert space defined as the completion of the space $C^\infty(M,\mbs(M))$ with respect to $(\cdot,\cdot)_2:=\int_M(\cdot,\cdot)_\ig d\vol_{\ig}$, where $(\cdot,\cdot)_\ig$ is the hermitian product on $\mbs(M)$ induced from $\ig$. The classical spectral theory and the self-adjointness of the Dirac operator in $L^2(M,\mbs(M))$ imply that $D_\ig$ has discrete real spectrum with finite multiplicities. Moreover the eigenvalues tend to both $+\infty$ and $-\infty$.

For any metric $\tilde\ig=e^{2u}\ig\in[\ig]$, we obtain a Dirac operator $D_{\tilde\ig}$. We denote the smallest (i.e. first) positive eigenvalue of $D_{\tilde\ig}$ by $\lm_1^+(\tilde\ig)$. Then the {\it B\"ar-Hijazi-Lott invariant} of $(M,\ig,\sa)$ is a non-negative real number $\lm_{min}^+(M,\ig,\sa)$ defined by (see for instance \cite{Ammann})
\[
\lm_{min}^+(M,\ig,\sa):=\inf_{\tilde\ig\in[\ig]}\lm_1^+(\tilde\ig)\Vol(M,\tilde\ig)^{\frac1m}.
\]
It is clear that the expression on the right-hand-side is chosen so as to remain conformally invariant, in a similar way as the Yamabe invariant \eqref{Yamabe invariant}. 

Finding bounds for this conformal invariant has attracted much interest during the last decades. A non-exhaustive list is \cite{Ammann2003, AGHM, AHM, Bar92, Hij86, Lott}. It can be seen from B\"ar's inequality \cite{Bar92} and from Hijazi's inequality \cite{Hij86} that
\[
\lm_{min}^+(M,\ig,\sa)^2\geq \begin{cases}
	2\pi\chi(M) & \text{for } m=2, \\[0.5em]
	\frac{m}{4(m-1)}Y(M,\ig) & \text{for } m\geq3,
\end{cases}
\]
where $\chi(M)$ is the Euler characteristic for Riemann surfaces.  When $M=S^m$, $m\geq2$, this implies that the B\"ar-Hijazi-Lott invariant is strictly positive. For general positivity, Lott \cite{Lott} derived the existence of a positive lower bound for $\lm_{min}^+(M,\ig,\sa)$ as long as the Dirac operator is invertible for some metric in the conformal class. His estimate does not require that $Y(M,\ig)>0$, but needs the weaker assumption that $D_\ig$ is invertible. An important improvement of this estimate was obtained by Ammann \cite{Ammann2003}. In particular, he showed that Lott's result extends to the case that the Dirac operator has non-trivial kernel, and hence
\[
\lm_{min}^+(M,\ig,\sa)>0
\]
for any closed Riemannian spin manifold $(M,\ig,\sa)$.
	
From the variational point of view, for a section $\psi\in C^\infty(M,\mbs(M))$, we can introduce a  functional in the form of the Rayleigh quotient:
\begin{\equ}\label{Rayleigh quotient}
J(\psi)=\frac{\Big(\int_M|D_{\ig}\psi|_\ig^{\frac{2m}{m+1}}d\vol_{\ig}\Big)^{\frac{m+1}{m}}}{\int_M(D_{\ig}\psi,\psi)_\ig d\vol_{\ig}}.
\end{\equ}
Based on some idea from \cite{Lott}, Ammann proved in \cite{Ammann} that 
\begin{\equ}\label{BHL-invariant}
\lm_{min}^+(M,\ig,\sa)=\inf_\psi J(\psi)
\end{\equ}
where the infimum is taken over the set of all smooth spinor fields for which
\[
\int_M(D_{\ig}\psi,\psi)_\ig d\vol_{\ig}> 0.
\]
Then upper bound estimates can be obtained by constructing suitable test spinors and evaluating the corresponding Rayleigh quotients in \eqref{Rayleigh quotient}. In a similar way as for the Yamabe invariant, the B\"ar-Hijazi-Lott invariant cannot be greater than that of the sphere: for $m\geq2$,
\begin{\equ}\label{spinorial Aubin inequ}
\lm_{min}^+(M,\ig,\sa)\leq\lm_{min}^+(S^m,\ig_{S^m},\sa_{S^m})=\frac m2 \om_m^{\frac1m}
\end{\equ}
where $\sa_{S^m}$ stands for the unique spin structure on $S^m$. The proof relies on a suitable cut-off argument performed on Killing spinors on the gluing of a sphere with large radius to the manifold, see \cite{Ammann2003, AGHM} for details. Then we are led to the following open problem

\begin{OProb*} 
Whether or not \eqref{spinorial Aubin inequ} is a strict inequality when $(M,\ig,\sa)$ is not conformally equivalent to the standard sphere $(S^m,\ig_{S^m},\sa_{S^m})$ of the same dimension?
\end{OProb*}

If the dimension of $M$ is $2$, there is a useful criterion for proving the strict inequality in \eqref{spinorial Aubin inequ}. The idea is based on the spinorial Weierstrass representation (see for instance \cite{KS96, Bar98, Fredrich98}). Examples of Riemann surfaces whose B\"ar-Hijazi-Lott invariant is strictly smaller than $2\sqrt{\pi}$ can be found in \cite{Ammann, Ammann2009}, see also \cite[Section 7]{SX2020} for a completely different approach on the two-dimensional torus. Apart from these, much less is known for higher dimensions. In their paper \cite{AHM},  Ammann, Humbert and Morel attacked this problem in the locally conformally flat setting. More precisely, if the spin manifold $(M,\ig,\sa)$ is  conformally flat around some point $p\in M$ and has an invertible Dirac operator $D_\ig$ and if, a further datum, the so-called {\it mass endomorphism} (a metric-induced self-adjoint endomorphism field of the spinor bundle, see \cite[Def. 2.10]{AHM}) has a non-zero eigenvalue at $p$, then \eqref{spinorial Aubin inequ} is a strict inequality. At this point one should be aware that the mass endomorphism of $(S^m,\ig_{S^m},\sa_{S^m})$ vanishes and that it does not characterize the round sphere since flat tori $\mathbb{T}^m$ also have vanishing mass endomorphism. We refer to \cite{AHM} for more details. In order to describe the dependence of the mass endomorphism on the Riemannian metrics, let $\cm_{U,\text{flat}}(M)$ be the set of all Riemannian metrics $\ig$ on $M$, which are flat on an open subset $U \subsetneq M$, and $\cm_{U,\text{flat}}^{\text{inv}}(M)\subset \cm_{U,\text{flat}}(M)$  the subset of metrics with invertible Dirac operators. Then it was proven in \cite{ADHH}, see also \cite{Hermann10}, that for dimension $m\geq3$, the subset $\cm_{p,U,\text{flat}}^{\neq0}(M)\subset \cm_{U,\text{flat}}^{\text{inv}}(M)$ of all Riemannian metrics with non-vanishing mass endomorphism at $p\in U$ is dense in $\cm_{U,\text{flat}}(M)$ with respect to the $C^\infty$-topology provided that $\cm_{U,\text{flat}}^{\text{inv}}(M)\neq\emptyset$ (it has been proved in \cite{Maier97} that $\cm_{U,\text{flat}}^{\text{inv}}(M)\neq\emptyset$ for any closed spin $3$-manifold, and for general cases, the $\al$-genus on the spin bordism class of $M$ describes whether the set $\cm_{U,\text{flat}}^{\text{inv}}(M)$ is admissible or not, see \cite{ADH}). However, for an arbitrarily given metric $\ig$ on $M$, examining the non-vanishing of its associated mass endomorphism (without any perturbation) is still a difficult issue.

We note that the proofs in \cite{AHM} and \cite{Hermann10} were given under the condition that $(M,\ig,\sa)$ is conformally flat in a neighborhood $U$ of a point $p\in M$, particularly, the behavior of $\ig$ outside~$U$ is rather irrelevant. The results obtained there describe that, somehow, the  non-vanishing of the mass endomorphism at $p\in U$ is induced from a perturbation of the metric $\ig$ outside $U$. In a very recent paper \cite{Is-Xu}, the authors constructed a family of non-locally conformally flat metrics $\ig_\vr$, $\vr>0$, on the $m$-sphere $S^m$ equipped with its unique spin structure $\sa_{S^m}$ such that the strict inequality
\begin{\equ}\label{inequality Is-Xu}
\lm_{min}^+(S^m,\ig_\vr,\sa_{S^m})<\lm_{min}^+(S^m,\ig_{S^m},\sa_{S^m})
\end{\equ}
holds for dimension $m\geq4$. The significance of this result is that the metric $\ig_\vr$ is a $C^\infty$-perturbation of the standard round metric $\ig_{S^m}$, i.e., $\ig_\vr\to\ig_{S^m}$ in $C^\infty$-topology on $S^m$ as $\vr\to0$, and $\ig_\vr|_U=\ig_{S^m}|_U$ for an open subset $U\subset S^m$. Comparing with the aforementioned results in \cite{AHM} and \cite{Hermann10}, it can be seen that even if the metric $\ig_\vr$ is conformally flat somewhere on $S^m$, the appearance of the non-locally conformally flat part is the key reason to obtain the strict inequality \eqref{inequality Is-Xu}. Moreover, the proof does not stand on the invertibility of the Dirac operators $D_{\ig_\vr}$, $\vr>0$.

It is well known that, for $m\geq4$, the conformal flatness of $(M,\ig)$ is characterized by the nullity of the Weyl tensor. However, what is the exact role of the Weyl tensor in the general inequality \eqref{spinorial Aubin inequ}  when $(M,\ig)$ is non-locally conformally flat remains unexplored.

By looking at the variational problem  \eqref{BHL-invariant}, we can see that any solution of the Euler-Lagrange equation can be rescaled to a solution of any of the following two equations
\[
D_\ig\psi=\lm_{min}^+(M,\ig,\sa)|\psi|_\ig^{\frac2{m-1}}\psi \quad \text{with } \int_M|\psi|_\ig^{\frac{2m}{m-1}}d\vol_\ig=1
\]
or (since $\lm_{min}^+(M,\ig,\sa)$ is always positive)
\[
 D_\ig\psi=|\psi|_\ig^{\frac2{m-1}}\psi \quad \text{with } \int_M|\psi|_\ig^{\frac{2m}{m-1}}d\vol_\ig=\lm_{min}^+(M,\ig,\sa)^m. 
\]
Hence, in order to resolve the aforementioned open problem, it is sufficient to find non-trivial solutions of
\begin{\equ}\label{Euler-Lagrange BHL}
	D_\ig\psi=|\psi|_\ig^{\frac2{m-1}}\psi \quad \text{with } \int_M|\psi|_\ig^{\frac{2m}{m-1}}d\vol_\ig<\big(\frac m2\big)^m\om_m.
\end{\equ}
In this paper we introduce a new variational approach into this problem and we give an affirmative answer for non-locally conformally flat Riemannian spin manifold with dimension at least~$6$. Hence our main result reads as
\begin{Thm}\label{main thm}
	Let $(M,\ig,\sa)$ be a closed spin manifold of dimension $m\geq6$. If $(M,\ig)$ is not locally conformally flat, then there exists a non-trivial solution to Eq. \eqref{Euler-Lagrange BHL}; moreover, there holds 
	\begin{\equ}\label{strict-inequality}
	\lm_{min}^+(M,\ig,\sa)<\lm_{min}^+(S^m,\ig_{S^m},\sa_{S^m})=\frac m2 \om_m^{\frac1m},
	\end{\equ}
where $\sa:P_{\Spin}(M)\to P_{\SO}(M)$ can be any choice of spin structure on $M$.
\end{Thm}

Note that we can approximate in the $C^\infty$-topology any given locally conformally flat metric on $M$ by a sequence of metrics, which are not conformally flat in a sufficiently small neighborhood of a point $p\in M$. Thus, as an immediate consequence of Theorem \ref{main thm}, we obtain

\begin{Cor}\label{main Cor}
		Let $(M,\ig,\sa)$ be a closed spin manifold of dimension $m\geq6$. The set of all Riemannian metrics for which the strict inequality \eqref{strict-inequality} holds is dense in the set of all Riemannian metrics on $M$ with respect to the $C^\infty$-topology.
\end{Cor}

\begin{Rem}
In this remark we will summarize the reasons, for which we are interested in the strict inequality \eqref{strict-inequality}. In fact, such a strict inequality has several applications:
	\begin{itemize}
		\item Inequality \eqref{strict-inequality} implies that the invariant $\lm_{min}^+(M,\ig,\sa)$ is attained by a generalized conformal metric $\tilde\ig=|f|^{\frac2{m-1}}\ig$, i.e., there exists a spinor $\tilde\va$ on $(M,\tilde\ig,\sa)$ such that
		\[
		D_{\tilde\ig}\tilde\va=\lm_{min}^+(M,\ig,\sa)\tilde\va, \quad |\tilde\va|_{\tilde\ig}\equiv1
		\] 
		on $M\setminus f^{-1}(\{0\})$. Here $f\in C^2(M)$ may have some zeros but  $\supp f=M$, that is, the zero set of $f$ does not contain any non-empty open set of $M$; see \cite[Chapter 4]{Ammann}.

		\item As was mentioned before, if $m\geq3$, the Hijazi inequality \cite{Hij86} relates $\lm_{min}^+(M,\ig,\sa)$ to the Yamabe invariant $Y(M,\ig)$ via
		\[
		\lm_{min}^+(M,\ig,\sa)^2\geq\frac{m}{4(m-1)}Y(M,\ig).
		\]
		Hence inequality \eqref{strict-inequality} implies 
		\[
		Y(M,\ig)<m(m-1)\om_m^{\frac2m},
		\]
		which played the crucial role in the solvability of the Yamabe problem. Therefore, our results Theorem \ref{main thm} and Corollary \ref{main Cor} provide an alternative proof on the existence of a metric with constant scalar curvature in the conformal class of $[\ig]$ in the spinorial setting, which is of particular interest in differential geometry.
		
		\item Last but not least, using inequality \eqref{strict-inequality} one improves the existence theory for the spinorial analogue of the Brezis-Nirenberg problem
		\begin{\equ}\label{spinorial B-N}
			D_\ig\psi=\mu\psi+|\psi|_\ig^{\frac2{m-1}}\psi, \quad \mu\in\R
		\end{\equ}
		on $(M,\ig,\sa)$. This problem has been studied in \cite{Isobe11}, and an improvement was made recently in  \cite{Ba-Xu-JFA21}. It follows from the existence results in \cite{Ba-Xu-JFA21} that, for every $\mu>0$, a ground-state solution $\psi_\mu$ (which has the smallest energy among all solutions) of Eq. \eqref{spinorial B-N} always exists, and particularly, different values of $\mu$ correspond to distinguished  solution branches bifurcating from positive eigenvalues $\lm_k^+(\ig)$'s of the Dirac operator $D_\ig$. By virtue of inequality \eqref{strict-inequality} and the variational approach developed in \cite{Ba-Xu-JFA21}, we can obtain the existence result for $\mu=0$ and we can also see a new phenomenon that the solution branch bifurcating from the first positive eigenvalue $\lm_1^+(\ig)$ can be continued to $\mu<0$ in the sense that the energy changes continuously. Setting $\cs_k$, $k=1,2,\dots$, the energy branches bifurcating from the $k$-th positive eigenvalue of $D_\ig$, we can visualize in Figure 1 the updated existence theory. Here we omit the specific proofs since it is of independent interest, and we refer the readers to \cite{Ba-Xu-JFA21} for more details about Eq. \eqref{spinorial B-N}.
		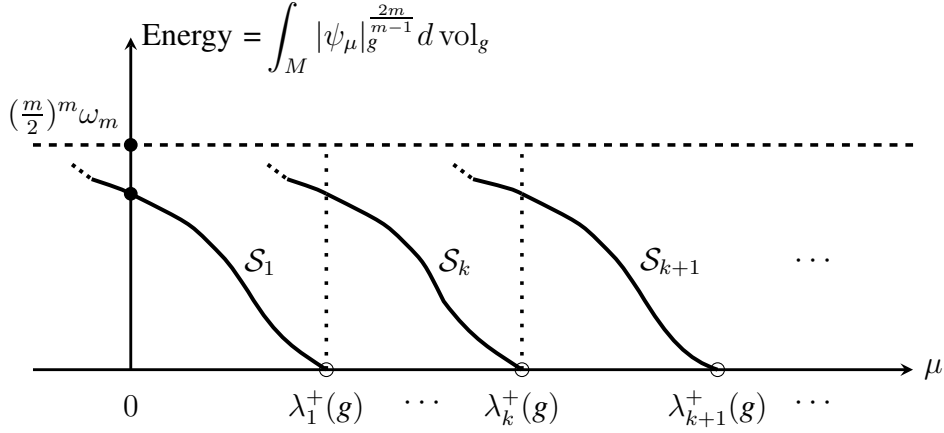
\begin{figure}[ht]	\begin{center}
			\begin{tikzpicture}[scale=1.3]
				\draw[thick,->,>=stealth,line width=0.2ex] (-1,0) -- (8,0);
				\draw[thick,->,>=stealth,line width=0.2ex] (0,0) -- (0,3.4);
				\draw[dashed,line width=0.25ex] (-1,2.3) -- (8,2.3);
				\draw[loosely dotted,line width=0.25ex] (2,0) -- (2,2.3);
				\draw[loosely dotted,line width=0.25ex] (4,0) -- (4,2.3);
				\coordinate [label=135: $(\frac m2)^m\om_m$] (cc) at (0,2.3);
				\draw (cc);
				\coordinate [label=-90:$0$] (o) at (0,-0.15);
				\draw (o);
				\coordinate [label=-90:$\lm_1^+(\ig)$] (a) at (2,-0.1);
				\draw (a);
				\coordinate [label=-90:$\cdots$] (z) at (3,-0.15);
				\draw (z);
				\coordinate [label=-90:$\lm_k^+(\ig)$] (b) at (4,-0.1);
				\draw (b);
				\coordinate [label=-90:$\lm_{k+1}^+(\ig)$] (c) at (6,-0.1);
				\draw (c);
				\coordinate [label=-90:$\cdots$] (zz) at (7,-0.15);
				\draw (zz);
				\coordinate [label=0:$\mu$] (e) at (8,0);
				\draw (e);
				\coordinate [label=0:$\displaystyle\text{Energy\,\,=}\int_M|\psi_\mu|_\ig^{\frac{2m}{m-1}}d\vol_\ig$] (f) at (0,3.4);
				\draw (f);
				\foreach \p in {a,b,c} \draw (\p)++(0,0.1) circle (2pt);
				\foreach \p in {cc} \fill (\p) circle (2pt);
				\draw[style=dotted,line width=1.5pt] (-0.6,2.1) -- (-0.4, 1.95);
				\draw[line width=1.5pt] (-0.4, 1.95) .. controls (-0.1, 1.85) .. (0,1.8);
				\draw[line width=1.5pt] (0,1.8) .. controls (0.6,1.5)
				.. (0.8,1.3) .. controls (1,1.1)
				.. (1.2,0.8) .. controls (1.5,0.3) and (1.7,0.2) .. (2, 0);
				\coordinate [label=-180: $\ $] (cc0) at (0,1.5);
				\draw (cc0);
				\foreach \p in {cc0} \fill (\p)++(0,0.3) circle (2pt);
				\coordinate [label=-180: $\cs_1$] (s1) at (1.6,1.1);
				\draw (s1);
				\draw[style=dotted,line width=1.5pt] (1.4,2.1) -- (1.6, 1.95);
				\draw[line width=1.5pt] (1.6, 1.95) .. controls (1.9, 1.85) .. (2,1.8);
				\draw[line width=1.5pt] (2,1.8) .. controls (2.6,1.5)
				.. (2.8,1.3) .. controls (3,1.1)
				.. (3.2,0.7) .. controls (3.5,0.3) and (3.7,0.2) .. (4, 0);
				\coordinate [label=-180: $\cs_k$] (sk) at (3.6,1.1);
				\draw (sk);
				\draw[style=dotted,line width=1.5pt] (3.3,2.1) -- (3.5, 1.95);
				\draw[line width=1.5pt] (3.5, 1.95) .. controls (3.9, 1.85) .. (4,1.8);
				\draw[line width=1.5pt] (4,1.8) .. controls (4.6,1.5)
				.. (4.8,1.3) .. controls (5,1.1)
				.. (5.2,0.8) .. controls (5.5,0.3) and (5.7,0.1) .. (6, 0);
				\coordinate [label=-180: $\cs_{k+1}$] (skk) at (5.95,1.1);
				\draw (skk);
				\coordinate [label=-90:$\cdots$] (zzz) at (7,1.3);
				\draw (zzz);
			\end{tikzpicture}
			\caption{\it Ground-state energy branch $\cs_1$ is continued to $\mu<0$}
	\end{center}\end{figure}
	\end{itemize}
\end{Rem}

\begin{Rem}
In case the dimension $m\not\equiv 3$ mod $4$, due to the symmetries of the Dirac spectrum the inequality \eqref{strict-inequality} is equivalent to 
\begin{\equ}\label{strict-inequality-negative}
	\lm_{min}^-(M,\ig,\sa)=\inf_{\tilde\ig\in[\ig]}|\lm_1^-(\tilde\ig)|\Vol(M,\tilde\ig)^{\frac1m}<\frac m2 \om_m^{\frac1m}
\end{\equ}
where $\lm_1^-(\tilde\ig)$ stands for the largest negative eigenvalue of $D_{\tilde\ig}$ with $\tilde\ig\in[\ig]$. At this moment, we actually have
\[
\lm_{min}^-(M,\ig,\sa)=\lm_{min}^+(M,\ig,\sa).
\]
If $m\equiv3$ mod $4$, the spectrum of the Dirac operator is no longer symmetric  (see for instance \cite[Example 2.15]{AHM} where the real projective space $\R P^{4k+3}$ was studied). The validity of \eqref{strict-inequality-negative} can be then turned into the study of the ``negatively signed" Dirac equation
\[
D_\ig\psi=-|\psi|_\ig^{\frac2{m-1}}\psi \quad \text{with} \quad 
\int_M|\psi|_\ig^{\frac{2m}{m-1}}d\vol_\ig<\big(\frac m2\big)^m\om_m
\]
and vice versa.
\end{Rem}

From the analytical view point,  $\psi\in C^1(M,\mbs(M))$ is a solution to Eq. \eqref{Euler-Lagrange BHL} if and only if $\psi$ is a critical point of the functional 
\begin{\equ}\label{the functional}
\Phi(\psi)=\frac12\int_M(D_\ig\psi,\psi)_\ig d\vol_{\ig}-\frac{m-1}{2m}\int_M|\psi|_\ig^{\frac{2m}{m-1}}d\vol_{\ig}.
\end{\equ}
Note that the spectrum of the Dirac operator $D_\ig$ consists of an infinite sequence of eigenvalues tending to both $+\infty$ and $-\infty$, the functional $\Phi$ is of strongly indefinite type (its leading part is respectively coercive and anti-coercive on infinite-dimensional subspaces of the energy space). And hence, a critical point of $\Phi$ has infinite Morse index and infinite co-index. Intuitively, this variational structure is much more complicated than the Rayleigh quotient given in \eqref{Rayleigh quotient} and the associated minimizing characterization \eqref{BHL-invariant}. Here, our choice to study \eqref{the functional} instead of \eqref{Rayleigh quotient}-\eqref{BHL-invariant} is based on two major reasons:
\begin{itemize}
	\item[1.]  Looking at the functional $J$ in \eqref{Rayleigh quotient}, one sees that the exponent $\frac{2m}{m+1}\in(1,2)$ in the numerator is not an integer for any $m\geq2$. Therefore, the second Fr\'echet derivative of $J$ does not exist in general, particularly, it can not exist if the variable $\psi$ vanishes on an open subset of $M$. In other words, the functional $J$ is not smooth enough to be approximated by the Taylor expansion with higher order terms.

	\item[2.] In a recent work of the second author with T. Isobe \cite{Is-Xu}, a family of non-locally conformally flat metrics has been constructed on the $S^m$, $m\geq4$, so that the strict inequality \eqref{strict-inequality} holds true. This result provides the introductory examples of spin manifolds that justify the validity of  \eqref{strict-inequality}, and the approach in \cite{Is-Xu} is based on bifurcation techniques and the explicit expression of the functional $\Phi$ in \eqref{the functional}. Motivated by this work, we tend to use the functional $\Phi$ as it may better reflect the influence of non-locally conformally flatness in the energy estimates.
%	
%	
%	The proof of our theorem relies on constructing a suitable test spinor, which is inspired by the resolution of the Yamabe problem. However, if we directly use the function $J$ in \eqref{Rayleigh quotient}, the detailed computations in Section \ref{sec: proof the main results} later would show that: for dimension $m\geq6$ and some suitably chosen test spinor $\bar\va_\vr\in C^\infty(M,\mbs(M))$, $\vr>0$, there hold
%		\[
%		\Big(\int_M|D_{\ig}\bar\va_\vr|_\ig^{\frac{2m}{m+1}}d\vol_{\ig}\Big)^{\frac{m+1}{m}}=\big( \frac m2 \big)^{m+1}\om_m^{\frac{m+1}{m}} + o(\vr^4) 
%		\]
%		and
%		\[
%		\int_M(D_{\ig}\bar\va_\vr,\bar\va_\vr)_\ig d\vol_{\ig}=\big( \frac m2 \big)^{m}\om_m -C\vr^4 + o(\vr^4) 
%		\]
%		for positive constant $C>0$. And hence we have $J(\bar\va_\vr)>\frac m2\om_m^{1/m}$ for all small $\vr$, which does not imply the inequality \eqref{strict-inequality}.
\end{itemize} 
%It is then natural to think that the functional $J$ in \eqref{Rayleigh quotient} is not a good choice of variational characterization for the problem.

In order to prove Theorem \ref{main thm}, our strategy is to make a careful analysis of the functional $\Phi$ in \eqref{the functional} which is twice differentiable on $H^{1/2}(M,\mbs(M))$, the Sobolev space of $1/2$-order. As was mentioned before,  the proof relies on constructing a suitable test spinor $\bar\va_\vr\in C^{\infty}(M,\mbs(M))$. However, since the functional $\Phi$ is strongly indefinite, critical points of $\Phi$ are generally saddle points. In this way, it is not enough to simply evaluate $\Phi(\bar\va_\vr)$. One needs more structural information of $\Phi$ in a small neighborhood of $\bar\va_\vr$. To do this, we construct a necessary constraint for $\Phi$, the Nehari-type manifold $\msn$. Every critical point of $\Phi$ has to lie on $\msn$ and $\Phi|_\msn$ is bounded from below. Hence minimizers of $\Phi$ on $\msn$ may be called least energy solutions, or ground states, of Eq. \eqref{Euler-Lagrange BHL}. In order to prove $\inf_\msn\Phi<\frac m2\om_m^{1/m}$, we construct our test spinor $\bar\va_\vr$  via spinorial techniques provided by Bourguignon and Gauduchon \cite{BG}. Since these spinors $\bar\va_\vr$, $\vr>0$, do not lie on $\msn$, we have to find modifications of them that lie on the Nehari-type manifold, and we have to control the energy of these modifications. This is the main technical difficulty that we have to overcome. Another difficulty comes from the appearance of the Dirac-harmonic spinors, that is, the non-trivial kernel of the Dirac operator $D_\ig$. It can be seen that the elements in $\ker(D_\ig)$ have an influence on the second integration in \eqref{the functional}. In this case, we use a new idea, replacing the test spinor $\bar\va_\vr$ by $\bar\va_\vr-T(\bar\va_\vr)$ where $T(\psi)$ stands for the best approximation of an element $\psi\in L^{\frac{2m}{m-1}}(M,\mbs(M))$ in $\ker(D_\ig)$ with respect to the $L^{\frac{2m}{m-1}}$-norm.

\medskip

The article is organized as follows. In Sect. \ref{sec: functional settings}, we will give some preliminary materials on the functional $\Phi$. This section contains two parts. Firstly, we recall the $H^{1/2}$-spinors and the variational settings that are needed for $\Phi$ to be defined. Secondly, in an abstract setting, we prove an energy estimate around the ground state critical level of a specific class of strongly indefinite functionals, which includes $\Phi$ as a special case. This kind of energy estimate is a continuation of previous work \cite{Xu-CAG} by the second author, and we present here an abstract version so that it can apply to a wider range of situations. In Sect. \ref{sec: geometric preliminaries}, we recapture a specific local trivialization of the spinor bundle, called the Bourguignon-Gauduchon-trivialization. And for latter use, we adapt the detailed expansion of the Dirac operator in this trivialization. In the following, i.e. Sect. \ref{sec: proof the main results}, we give the complete proof of our Theorem \ref{main thm}. The detailed proof is divided into three major steps: 
\begin{itemize}
	\item[1.] The construction of a ``good" test spinor. In fact, under stereographic projection, a Killing spinor on the $m$-sphere yields a solution to $D_{\ig_{\R^m}}\psi=|\psi|_{\ig_{\R^m}}^{{2}/{m-1}}\psi$ on the Euclidean space $\R^m$. This solution will be rescaled, cut-off and later transplanted to a neighborhood of a given point $p_0\in M$. Some fundamental estimates are derived from elementary computations using this test spinor.
	
	\item[2.] The proof for the case there is no Dirac-harmonic spinors, that is $\ker(D_\ig)=\{0\}$.
	
	\item[3.] The proof the case $\ker(D_\ig)\neq\{0\}$. 
\end{itemize}
Both of the last two steps are the consequences of the abstract result developed in Sect. \ref{sec: functional settings}.

\section{Analytical preliminaries}\label{sec: functional settings}

In this section, we assume $(M,\ig,\sa)$ is a general closed Riemannian spin manifold, we fix the Riemannian metric $\ig$ and the spin structure $\sa$.

\subsection{Functional settings}

To treat the Eq. \eqref{Euler-Lagrange BHL} from a variational point of view, it is necessary to set up a functional framework. A suitable function space  $H^{1/2}(M,\mbs(M))$ of spinor fields is introduced, for example, in \cite{Isobe11, Isobe13}. For being self-contained, we give the definitions as follows.

Recall that the Dirac operator $D_\ig$ is self-adjoint on $L^2(M, \mbs(M))$ and has compact resolvents (see \cite{Friedrich00, LM}). We obtain a complete orthonormal basis $\psi_1,\psi_2, . . .$ of the Hilbert space $L^2(M, \mbs(M))$ consisting of the eigenspinors of $D_\ig$: $D_\ig\psi_k = \lm_k\psi_k$. Moreover, $|\lm_k| \to\infty$ as $k \to\infty$.

For $s\geq0$, we define the (unbounded) operator $|D_\ig|^s:L^2(M, \mbs(M))  \to L^2(M, \mbs(M))$ by
\[
\psi=\sum_{k=1}^\infty\al_k\psi_k \quad \mapsto \quad |D_\ig|^s\psi=\sum_{k=1}^\infty |\lm_k|^s\al_k\psi_k,
\]
where the domain of $|D_\ig|^s$ is given by
\[
H^s(M,\mbs(M)):=\Big\{\psi=\sum_{k=1}^\infty\al_k\psi_k\in L^2(M,\mbs(M)):\,  \sum_{k=1}^\infty|\lm_k|^{2s}|\al_k|^2<\infty\Big\}.
\] 
We remark here that $H^s(M,\mbs(M))$ coincides with the usual Sobolev space of order $s$, that is $W^{s,2}(M,\mbs(M))$ (cf. \cite{Adams, Ammann}). Furthermore, let $P^0:L^2(M,\mbs(M))\to\ker(D_\ig)$ denote the $L^2$-orthogonal projector, we can equip $H^s(M,\mbs(M))$ with the inner product
\[
\inp{\psi}{\va}_{s,2}:=\real(|D_\ig|^s\psi,|D_\ig|^s\va)_2+\real(P^0\psi,P^0\va)_2
\]
and the induced norm $\|\cdot\|_{s,2}$. This inner product turns $H^s(M,\mbs(M))$ into a Hilbert space.  In this paper, we are mainly concerned with $s=\frac12$ and the space $H^\frac12(M,\mbs(M))$.

Clearly $E := H^{1/2}(M,\mbs(M))$ with the inner product $\inp{\cdot}{\cdot}=\inp{\cdot}{\cdot}_{1/2,2}$ has the orthogonal decomposition into three subspaces
\begin{\equ}\label{decomposition}
	E = E^+\op E^0\op E^-
\end{\equ}
where $E^\pm$ is the positive (resp.\ negative) eigenspace of $D_\ig$, and $E^0=\ker(D_\ig)$ is its kernel (which may be trivial). By $P^\pm:E\to E^\pm$ and $P^0:E\to E^0$ we denote the orthogonal projections. If it is clear from the context, we simply write $\psi^\pm=P^\pm(\psi)$ and $\psi^0=P^0(\psi)$ for $\psi\in E$. And throughout this paper, we shall denote $L^q:=L^q(M,\mbs(M))$ with the norm $\|\psi\|_{L^q}^q= \int_M |\psi|_\ig^qd\vol_{\ig}$ for $q\geq1$ and denote $2^*=\frac{2m}{m-1}$  the critical Sobolev exponent of the embedding $H^{1/2}(M,\mbs(M))\hookrightarrow L^q(M,\mbs(M))$ for $1\leq q\leq 2^*$. From now on, when no confusion can arise, we will write $\inp{\cdot}{\cdot}$ instead of $\inp{\cdot}{\cdot}_{1/2,2}$  for the inner product on $H^{1/2}(M,\mbs(M))$, and $\|\cdot\|$ will stand for $\|\cdot\|_{1/2,2}$.

In this setting, the functional $\Phi$ in \eqref{the functional} can be reformulated as
\begin{\equ}\label{the functional in E}
\Phi(\psi)=\frac12\big( \|\psi^+\|^2-\|\psi^-\|^2 \big)-\frac1{2^*}\|\psi\|_{L^{2^*}}^{2^*} \quad \text{for } \psi\in E.
\end{\equ}
And, by elementary calculations, we have $\Phi$ is of $C^2$ (that is the second Fr\'echet derivative exists) and
\[
\inp{\nabla\Phi(\psi)}{\va}=\inp{\psi^+}{\va^+}-\inp{\psi^-}{\va^-}-\real\int_M|\psi|_\ig^{2^*-2}(\psi,\va)_\ig d\vol_{\ig}
\]
and
\[
\aligned
\inp{\nabla^2\Phi(\psi)[\phi]}{\va}&=\inp{\phi^+}{\va^+}-\inp{\phi^-}{\va^-}-\real\int_M|\psi|_\ig^{2^*-2}(\phi,\va)_\ig d\vol_{\ig}  \\
&\quad -(2^*-2)\int_M|\psi|_{\ig}^{2^*-4}\real(\psi,\phi)_\ig\real(\psi,\va)_\ig d\vol_{\ig}
\endaligned
\]
for $\psi,\phi,\va\in E$.
The quadratic part of $\Phi$ is of strongly indefinite type and, particularly, a critical point $\Phi$ has infinite Morse index and infinite co-index since both $E^\pm$ are infinite dimensional.  There is, in addition, one further point that is the lack of compactness for the functional $\Phi$ due to the non-compact embedding $H^{1/2}(M,\mbs(M))\hookrightarrow L^{2^*}(M,\mbs(M))$. In other words, this means that the functional $\Phi$ does not satisfy the Palais-Smale condition in the total energy range $(-\infty,+\infty)$. Such a compactness condition was originally introduced by Palais and Smale in \cite{PS} and has proved to be extremely useful in variational methods. To describe a compactness condition of this type, it is usually convenient to recall the following concept
\[
\aligned
&\text{A sequence $\{\psi_n\}\subset E$ is a $(PS)$-sequence for $\Phi$ if $\Phi(\psi_n)$ is bounded,}\\
&\text{uniformly in $n$, while $\nabla\Phi(\psi_n)\to0$ as $n\to\infty$.}
\endaligned
\]
And if $\Phi(\psi_n)\to c\in\R$ and $\nabla\Phi(\psi_n)\to0$, then $\{\psi_n\}$ is called a $(PS)_c$-sequence. In terms of this definition, the compactness condition may be then phrased as: any $(PS)$-sequence (or $(PS)_c$-sequence) has a (strongly) convergent subsequence. It is clear that if a $(PS)$-sequence of $\Phi$ converges to $\psi$, then $\psi$ is a critical point of $\Phi$. (Let us remind the readers here that the above definitions can be easily carried over to any $C^1$ functional on a Banach space).

The following lemma shows that $\Phi$ satisfies a local version of Palais-Smale condition (its proof can be found in \cite[Section 5]{Isobe11}). 

\begin{Lem}\label{PS Dirac}
	Let $\ga_{crit}=\frac1{2m}\big(\frac m2\big)^m \om_m$, where $\om_m$ stands for the volume of the standard sphere $S^m$. Then the functional $\Phi$ satisfies the Palais-Smale condition in the energy range $(-\infty, \ga_{crit})$ in the sense that any $(PS)_c$-sequence for $\Phi$ with $c<\ga_{crit}$ is compact in $E$ and $c$ is a critical value.
\end{Lem} 

\medskip

\begin{Rem}\label{remark 1}
Thanks to the explicit formulation of $\Phi$ in \eqref{the functional in E}, if $\psi\in E$ is a critical point of $\Phi$, then we have
\[
\Phi(\psi)=\Phi(\psi)-\frac12\inp{\nabla\Phi(\psi)}{\psi}=\Big(\frac12-\frac1{2^*}\Big)\int_M|\psi|_\ig^{2^*}d\vol_{\ig}=\frac1{2m}\int_M|\psi|_\ig^{2^*}d\vol_{\ig}.
\]
Hence all possible critical values of $\Phi$ are non-negative and, in particular, if $0<c<\ga_{crit}$ is a critical value of $\Phi$ and $\psi$ be the corresponding critical point then 
\[
0<\int_M|\psi|_\ig^{2^*}d\vol_{\ig}=2mc<\big(\frac m2\big)^m\om_m.
\]
Therefore, to prove Theorem \ref{main thm}, it is sufficient to find a specific critical value of $\Phi$ which stays below the threshold $\ga_{crit}$.
\end{Rem}

\subsection{A local energy estimate via reduction techniques}

In this subsection, we intend to establish an approximate calculation for the lowest/smallest minimax level of $\Phi$ so that one can insert test elements to obtain an accurate upper bound estimate. In order to do this, we are going to prove an abstract result for a strongly indefinite functional $\cl$ on a general Hilbert space $\ch$, the underlying idea consisting of doing separate studies of the ``geometry" of $\cl$ and its approximate critical sequences. Let us mention here that our abstract result here is a continuation of the work \cite{Xu-CAG}, and can be applied to more general situations.

In what follows, let us assume that $\ch$ is equipped with the scalar product $\inp{\cdot}{\cdot}$ and the induced norm $\|\cdot\|$. Let $\ch$ be decomposed into $\ch=X\op Y$ with $X$ and $Y$ being infinite dimensional orthogonal subspaces of $\ch$. Then each element $z\in\ch$ possesses the representation $z=z^X+z^Y$, where $z^X\in X$ and $z^Y\in Y$ are projections of $z$ in $X$ and $Y$ respectively. In this setting, let us consider the functional
\[
\cl(z)=\frac12\big( \|z^X\|^2-\|z^Y\|^2 \big)-\Psi(z)
\]
with $\Psi\in C^2(\ch,\R)$, a twice continuously Fr\'echet differentiable term, satisfying
\begin{itemize}
	\item[(H1)] There is a continuous function $\varrho:[0,\infty)\to[0,\infty)$ such that 
	$\normm{\nabla^2 \Psi(z)}\leq \varrho(\|z\|)$,
	where $\normm{\cdot}$ denotes the operator norm for continuous linear operators on $(\ch,\inp{\cdot}{\cdot})$.
	
	\item[(H2)] $\Psi(0)=0$ and there exists $p>2$ such that $\inp{\nabla\Psi(z)}{z}\geq p\Psi(z)$ for all $z\in\ch$.
	
	\item[(H3)] $\Psi\not\equiv0$ on $\ch$ and is convex.
	
	\item[(H4)] There exist $K>0$ and $\mu\in(\frac12,1)$ such that $\|\nabla\Psi(z)\|\leq K\inp{\nabla\Psi(z)}{z}^\mu+\mu\|z\|$ for all $z\in\ch$.
	
	\item[(H5)] There exists $\ka>1$ such that, for arbitrary $z,w\in\ch$, 
	\[
	\inp{\nabla^2 \Psi(z)[z+w]}{z+w}-2\inp{ \nabla\Psi(z)}{w}   
	\geq \ka\inp{ \nabla\Psi(z)}{z}.
	\]
\end{itemize}

Clearly, the functional $\cl$ has a trivial critical point at the origin. It can be seen from (H2) that the function $\Psi(tz)/t^p$, $t>0$, is nondecreasing in $t$ for all $z\in\ch$. This, combined with the convexity in (H3) and the requirement $\nabla\Psi(0)=0$ in (H4), it follows that $\Psi\geq0$ (i.e. $0$ is the global minimum of $\Psi$) and 
\[
\Psi(z)=0  \Longrightarrow \Psi(tz)=0 \text{ for all } t\in[0,1].
\] 
Furthermore, from (H4) and the obvious inequality
\[
ab\leq \mu a^{\frac1\mu}+(1-\mu)b^{\frac1{1-\mu}}, \quad \text{for } a,b\geq0
\]
we deduce that
\[
\|\nabla\Psi(z)\|\leq \epsilon\mu\|\nabla\Psi(z)\|+ \epsilon^{-\frac{\mu}{1-\mu}} (1-\mu)K^{\frac1{1-\mu}} \|z\|^{\frac\mu{1-\mu}} + \mu\|z\|
\]
where $\epsilon>0$ is arbitrary. If $\epsilon$ is fixed small, then we have the following boundedness of $\nabla\Psi$:
\begin{\equ}\label{N bdd}
	\|\nabla\Psi(z)\|\leq\frac{\epsilon^{-\frac{\mu}{1-\mu}} (1-\mu)K^{\frac1{1-\mu}}}{1-\epsilon\mu}\|z\|^{\frac\mu{1-\mu}} + \frac{\mu}{1-\epsilon\mu}\|z\| \quad \text{for all }z\in\ch.
\end{\equ}
Beyond the positivity and the convexity of $\Psi$,  (H5) implies that $\nabla^2 \Psi(z)\neq 0$ when $\Psi(z)>0$.

In what follows, we begin with the study of the ``geometry" of $\cl$ by performing a generalized version of the Lyapunov-Schmidt reduction procedure, developed for strongly indefinite functionals. This enable us to study an equivalent reduced functional whose critical points are in one-to-one correspondence with those of $\cl$. This kind of idea can be traced back to \cite{Amann, CL} in the study of periodic solutions of Hamiltonian systems and nonlinear wave equations in one space dimension, and some later developments can be found in \cite{Ackermann, BJS, SW2010} for  semilinear elliptic equations. Here, our purpose is to point out that the variational structure of $\cl$ may be set up with much weaker conditions on $\Psi$ to ensure an energy estimate for a minimax value of $\cl$.  
By taking the advantage of the positiveness and convexity of $\Psi$, we have a very good geometric behavior of the functional $\cl$ in the following sense:
\begin{itemize}
	\item[(i)] Since $\Psi$ is non-negative, for a given $\va\in X$, the functional $\cl_\va$ defined on $Y$ by
	\[
	\cl_\va(w)=\frac12\|\va\|^2-\frac12\|w\|^2-\Psi(\va+w), \quad w\in Y
	\]
	is anti-coercive, i.e., $\cl_\va(w)\to-\infty$ as $\|w\|\to\infty$.
	
	\item[(ii)] Since $\Psi$ is convex, $\cl_\va$ is strictly concave, in other words, the quadratic form 
	\[
	\inp{\nabla^2\cl_\va(w)[\,\cdot\,]}{\,\cdot\,}: Y\times Y\to\R
	\]
	is strictly negative definite.
\end{itemize}
Having the above two properties available, for a fixed $\va\in X\setminus\{0\}$, we almost arrive at the point where we can try to maximize the functional $\cl$ on $\R^+\va\op Y$, $\R^+=(0,\infty)$. However, just assuming the hypothesis (H1)-(H5), the supremum of $\cl|_{\R^+\va\op Y}$ is not necessarily finite (for instance, if we take $\Psi(z)=\frac1p\|z^Y\|^p$, $p>2$, we shall have $\cl(t\va)\to+\infty$ as $t\to\infty$ for every $\va\in X\setminus\{0\}$). The situation will be more interesting and important if there exists some $\va\in X\setminus\{0\}$ such that $\sup\cl|_{\R^+\va\op Y}<+\infty$ and the supremum is attained. Remark that $\va$ is varying in the space $X\setminus\{0\}$, we can then restrict ourselves to the choice $\va\in S^X:=\{\va\in X:\, \|\va\|=1\}$ without changing the supremum of $\cl$ on $\R^+\va\op Y$. By virtue of \eqref{N bdd} in which we can fix $\epsilon\leq 1-\mu$, we deduce from the second part of (H2) that 
\begin{\equ}\label{Psi-al upper}
	\Psi(z)\leq\frac1p\inp{\nabla\Psi(z)}{z}\leq\frac1p\|z\|^2+C\|z\|^{\frac1{1-\mu}}
\end{\equ}
for all $z\in\ch$, where $C>0$ is a constant. Indeed, this is a direct consequence of $\frac{\mu}{1-\epsilon\mu}<\frac\mu{1-\epsilon}\leq 1$.
And so, by the fact $\frac1{1-\mu}>2$, we have the existence of $\tau>0$ such that
\begin{\equ}\label{linking1}
	\max_{z\in\R^+\va\op Y}\cl(z)\geq\max_{t>0}\cl(t\va)\geq \max_{t>0} \Big( \frac{p-2}{2p} t^2- C t^{\frac1{1-\mu}} \Big) \geq \tau
	\quad \text{for all } \va\in S^+.
\end{\equ}
Therefore, we have found a candidate min-max scheme for $\cl$ which is strictly positive:  first maximize the functional $\cl$ on $\R^+\va\op Y$ and then minimize with respect to $\va\in S^X$ (we will only be interested in the case where such a min-max value of $\cl$ is finite). 

As was mentioned before, for strongly indefinite problems, critical points are of saddle type in general. Among all these critical levels, to have a better understanding of the structural information about the ground state energy, an efficient way is to reduce the functional $\cl$ to a subspace which contributes mostly to its critical points. Specifically, let $P:\ch\to X$ denote the orthogonal projection onto $X$, that is $Pz=z^X$ for $z\in\ch$, we have the following reduction principle.

\begin{Prop}\label{reduction}
	\begin{itemize}
		\item[$(1)$] There exists $\bt\in C^1(X,Y)$ such that for $\va\in X$
		\[
		w\in Y,\ w\neq\bt(\va) \Longrightarrow
		\cl(\va+w)<\cl(\va+\bt(\va)),
		\]
		that is $\bt(\va)$ is the unique maximizer of $\cl_\va$ on $Y$ and hence
		\begin{\equ}\label{reduction equ}
			\bt(\va)=-(I-P)\nabla\Psi(\va+\bt(\va));
		\end{\equ}
		furthermore, for $\va\in\ch^+$,
		\begin{\equ}\label{reduction bdd}
			\|\bt(\va)\|^2\leq 2\Psi(\va)
		\end{\equ}
		and
		\begin{\equ}\label{reduction derivative bdd}
			\normm{\nabla\bt(\va)}\leq \normm{\nabla^2 \Psi(\va+\bt(\va))}.
		\end{\equ}
		
		\item[$(2)$]  Let $\cj: X\to\R$ be given by $\cj(\va)=\cl(\va+\bt(\va))$, then $\|\nabla \cj(\va)\|=\|\nabla\cl(\va+\bt(\va))\|$ for all $\va\in X$.
	\end{itemize}
\end{Prop}

\begin{proof}
	For a given $\va\in X$, since $\cl_\va$ is strictly concave on $Y$, we find that for any $w\in Y$ and any sequence $\{w_n\}$ in $Y$ such that $w_n\rightharpoonup w$ weakly in $Y$ there holds
	\begin{\equ}\label{upper-semi-c}
	\cl_\va(w)\geq\limsup_{n\to\infty}\cl_\va(w_n).
	\end{\equ}
	By virtue of this observation, thanks to the anti-coerciveness of $\cl_\va$, we can choose a bounded maximizing sequence $\{w_n\}$ of $\cl_\va$ so that $w_n\rightharpoonup w_\va$ weakly in $Y$ for some $w_\va$. Then we can use \eqref{upper-semi-c} to conclude that $w_\va$ is indeed a maximizer of $\cl_\va$. Notice that, for any $c\in\R$, the set $\{w\in Y:\, \cl_\va(w)\geq c\}$ is strictly convex due to the strict concaveness of $\cl_\va$ on $Y$. Hence, $w_\va$ is also the unique critical point of $\cl_\va$ (since any critical point of $\cl_\va$ must be a local maximizer, there can be no more than one of such point). In the sequel, let us write $\bt(\va)=w_\va\in Y$ for $\cl_\va$. Note that since $\bt(\va)$  satisfies $\inp{\nabla\cl_\va(\bt(\va))}{w}=0$ for all $w\in Y$, we then have
	\[
	\bt(\va)+(I-P)\nabla\Psi(\va+\bt(\va))=0.
	\]
	By taking derivative of the left hand side with respect to the slot of $\bt(\va)$, we get a self-adjoint operator on $Y$:
	\begin{\equ}\label{the operator}
		w\mapsto w+(I-P)\nabla^2 \Psi(\va+\bt(\va))[w], \quad \forall w\in Y.
	\end{\equ}
	Since
	\[
	\aligned
	&\inp{w+(I-P)\nabla^2 \Psi(\va+\bt(\va))[w]}{w} \\[0.3em]
	&\qquad = \inp{w+\nabla^2 \Psi(\va+\bt(\va))[w]}{w} \\[0.3em]
	&\qquad \geq
	\|w\|^2   \qquad \text{by the convexity of } \Psi 
	\endaligned
	\]
	for all $w\in Y$, we see that the operator in \eqref{the operator} is invertible and the operator norm of its inverse is bounded by $1$. Hence, the implicit function theorem shows $\bt\in C^1(X,Y)$ and
	\[
	\normm{\nabla\bt(\va)}\leq  \normm{\nabla^2 \Psi(\va+\bt(\va))}.
	\]
	Meanwhile, since $\bt(\va)$ maximizes $\cl_\va$, 
	\[
	\aligned
	0&\leq \cl_{\va}(\bt(\va))-\cl_\va(0) 
	=-\frac12\|\bt(\va)\|^2-\Psi(\va+\bt(\va))+\Psi(\va) \\[0.3em]
	&\leq -\frac12\|\bt(\va)\|^2+\Psi(\va) 
	\qquad \text{by the non-negativeness of }\Psi 
	\endaligned
	\]
	we conclude the boundedness of $\|\bt(\va)\|$. This finishes the proof of $(1)$.
	
	Finally, to check $(2)$, we first rewrite \eqref{reduction equ} as $(I-P)\nabla\cl(\va+\bt(\va))=0$. Note that for $\phi\in\ch$, we have
	\[
	\aligned
	&\inp{\nabla \cl(\va+\bt(\va))}{\phi}  \\[0.3em]
	&\qquad =\inp{\nabla \cl(\va+\bt(\va))}{P\phi}+\inp{\nabla \cl(\va+\bt(\va))}{(I-P)\phi} \\[0.3em]
	&\qquad =\inp{\nabla \cj(\va)}{P\phi}
	\endaligned
	\]
	which implies $\|\nabla \cj(\va)\|\geq \|\nabla\cl(\va+\bt(\va))\|$. And on the other hand, we find for $v\in X$
	\[
	\aligned
	\inp{\nabla \cj(\va)}{v}&=\inp{\nabla\cl(\va+\bt(\va))}{v}+\inp{\nabla\cl(\va+\bt(\va))}{\nabla\bt(\va)[v]} \\[0.3em]
	&=\inp{\nabla\cl(\va+\bt(\va))}{v}
	\endaligned
	\]
	which suggests $\|\nabla \cj(\va)\|\leq \|\nabla\cl(\va+\bt(\va))\|$. Hence we have  $\|\nabla \cj(\va)\|= \|\nabla\cl(\va+\bt(\va))\|$, and this completes the proof.
\end{proof}

It can be seen from Proposition \ref{reduction} that a critical point $\va$ of $\cj$ (if exists) induces a critical point of $\cl$ via the mapping $\va\mapsto \va+\bt(\va)$, and the corresponding values of $\cj$ and $\cl$ coincide. However, the reverse is not necessarily true in general cases unless one can show that the restricted functional $\cl|_{\R^+\va\op Y}$ has at most one (non-trivial) critical point for each $\va\in S^X$ (notice that the uniqueness of $\bt(\va)$ implies that a critical point of $\cl|_{\R^+\va\op Y}$ must be of the form $t\va+\bt(t\va)$ for some $t>0$, in this way, it is sufficient to show that the function $t\mapsto \cj(t\va)$ has only one critical point on $(0,+\infty)$ and this will be seen in Remark \ref{rem:uniquenss-Nehari} below). The advantage of Proposition \ref{reduction} is that it allows us to restrict our attention to the critical points of $\cj$ for which the subspace $Y$ collapse to the graph $\{(\va,\bt(\va)):\, \va\in X\}$. In the sequel, we call $\cj$ the reduced functional of $\cl$. And the following result is an immediate consequence of Proposition \ref{reduction}.

\begin{Cor}\label{reduction-corollary}
	If $\{\va_n\}\subset X$ is a $(PS)$-sequence for the reduced functional $\cj$
	then $\{\va_n+\bt(\va_n)\}$ is a $(PS)$-sequence for $\cl$.
\end{Cor}
\begin{proof}
Since we have $\cj(\va_n)=\cl(\va_n+\bt(\va_n))$ and $\|\nabla\cj(\va_n)\|=\|\nabla\cl(\va_n+\bt(\va_n))\|$, the conclusion follows directly from the definition of a $(PS)$-sequence.
\end{proof}

In what follows, the essence of our analysis is to consider the minimax value
\begin{\equ}\label{ga-al1}
	\ga=\inf_{\va\in X\setminus\{0\}}\max_{t>0}\cj(t\va).
\end{\equ}
Since $\cj(\va)\geq\cl(\va)$ for all $\va\in X$, together with \eqref{Psi-al upper} and \eqref{linking1},  we have the existence of $r>0$ such that
\[
\|\va\|\leq r \Longrightarrow 
\cj(\va)\geq0
\quad \text{and}\quad
\|\va\| = r \Longrightarrow 
\cj(\va)\geq\tau.
\]
It is clear that if there exists additionally $\phi\in X$, $\|\phi\|>r$, such that $\cj(\phi)<0$, then the reduced functional $\cj$ possesses the so-called mountain pass geometry (see \cite{AR, Willem}), and hence $\ga\in[\tau,+\infty)$ will be a candidate energy level of $\cj$ at which there might exist non-trivial critical points (just remind here that the functional $\cj$ could be lack of compactness, but it is still possible to use the deformation lemma \cite[Lemma 2.3]{Willem} to obtain a $(PS)_\ga$-sequence of $\cj$). The mountain pass geometry of $\cj$ will be fulfilled when the $\Psi$-term satisfies stronger assumptions, say for instance, $\Psi(z)$  vanishes if and only if $z=0$ (cf. \cite{Ackermann, BJS, SW2010}). Unfortunately, since the hypothesis (H1)-(H5) are rather weak (we allow $\Psi$ to vanish on certain part of $\ch$), we do not have a clear linking structure for $\cj$. Here, our strategy consists of finding critical point of $\cj$ on the associated Nehari-type manifold $\msn$ defined as
\[
\msn=\big\{ \va\in X\setminus\{0\}:\, \inp{\nabla \cj(\va)}{\va}=0 \big\}.
\]
Different from the usual concept of Nehari manifold, we emphasize here that the set $\msn$ above is not homeomorphic to the sphere $S^X$ in general. In fact we allow that there exists $\va\in X\setminus\{0\}$ such that $\Psi(t\va)\equiv 0$ for all $t>0$, then, by Proposition \ref{reduction}, we have $\bt(t\va)\equiv0$, and hence $\cj(t\va)=\frac{t^2}2\|\va\|^2$, i.e. $\cj(t\va)$ is a quadratic function and  have no critical points for $t\in[0,\infty)$ other than $t=0$. 

\begin{Lem}\label{lemma 1}
	If $\va\in\msn$, then $\Psi(\va)>0$ and $\Psi(\va+\bt(\va))>0$.
\end{Lem}
\begin{proof}
	Assume contrarily that $\Psi(\va)=0$, then it follows that $\nabla\Psi(\va)=0$ (since $0$ is the global minimum of $\Psi$). And from \eqref{reduction bdd}, we have $\bt(\va)=0$. Hence
	\[
	\inp{\nabla \cj(\va)}{\va}=\inp{\va-P \nabla\Psi(\va)}{\va}=\|\va\|^2>0,
	\]
	which is a contradiction. 
	
	If $\Psi(\va+\bt(\va))=0$, then it also follows that $\nabla\Psi(\va+\bt(\va))=0$. By \eqref{reduction equ}, we deduce that $\bt(\va)=0$. Thus we get a similar contradiction as above.
\end{proof}

Let us introduce a functional $\ck:X\to\R$ by
\begin{\equ}\label{K-al def}
	\ck(\va)=\inp{\nabla \cj(\va)}{\va}.
\end{\equ}
We have $\ck\in C^1(X,\R)$ and its derivative is given by
\[
\inp{\nabla \ck(\va)}{w}=\inp{\nabla \cj(\va)}{w}+\inp{\nabla^2 \cj(\va)[\va]}{w}
\]
for all $\va,w\in X$. Also, we have $\msn=\ck^{-1}(0)\setminus\{0\}$. 

\begin{Lem}\label{K-al estimate}
	For $\va\in X$,
	\[
	\inp{\nabla \ck(\va)}{\va}\leq 2\ck(\va)-(\ka-1)\inp{\nabla\Psi(\va+\bt(\va))}{\va+\bt(\va)},
	\]
	where $\ka>1$ is the constant in {\rm (H5)}.
\end{Lem}
\begin{proof}
	Since
	\begin{\equ}\label{D1}
	\aligned
	\inp{\nabla \cj(\va)}{\phi}&=\inp{\va-P \nabla\Psi(\va+\bt(\va))}{\phi}  \\[0.3em]
	&=\inp{\va}{\phi}-\inp{ \nabla\Psi(\va+\bt(\va))}{\phi}
	\endaligned
	\end{\equ}
	for all $\va,\phi\in X$, it follows directly that
	\begin{\equ}\label{X1}
		\inp{\nabla^2\cj(\va)[\phi]}{\phi}=\|\phi\|^2-\inp{ \nabla^2 \Psi(\va+\bt(\va))\big[\phi+\nabla\bt(\va)[\phi]\big]}{\phi}.
	\end{\equ}
	By taking derivative with respect to $\va$ in the relation \eqref{reduction equ}, we have
	\begin{\equ}\label{reduction equ derivative}
		 \nabla\bt(\va)[\va] =-(I-P)\nabla^2 \Psi(\va+\bt(\va))\big[\va+\nabla\bt(\va)[\va]\big]
	\end{\equ}
	for $\va\in X$. For ease of notations, we denote $z_\va=\va+\bt(\va)$ and $w_\va=\nabla\bt(\va)[\va]-\bt(\va)$. Then we can deduce from \eqref{reduction equ derivative} that
	\begin{\equ}\label{D2}
	\big\|\nabla\bt(\va)[\va]\big\|^2=-\inp{\nabla^2\Psi(z_\va)[z_\va+w_\va]}{\nabla\bt(\va)[\va]}.
	\end{\equ}
    Combining \eqref{D1}, \eqref{X1} and the relations \eqref{reduction equ} and \eqref{D2}, we obtain
	\begin{eqnarray*}
    %\aligned
	\inp{\nabla^2\cj(\va)[\va]}{\va}&\overset{\eqref{X1}}{=}&\|\va\|^2-\inp{ \nabla^2 \Psi(z_\va)[z_\va+w_\va]}{\va} \\[0.3em]
	&\overset{\eqref{D1}}{=}&\inp{\nabla \cj(\va)}{\va} + \inp{\nabla\Psi(z_\va)}{\va}-\inp{ \nabla^2 \Psi(z_\va)[z_\va+w_\va]}{\va} \\[0.3em]
	&\overset{\eqref{D2}}{=}&\inp{\nabla \cj(\va)}{\va}-\inp{ \nabla^2 \Psi(z_\va)[z_\va+w_\va]}{z_\va+w_\va} + \inp{\nabla\Psi(z_\va)}{\va}\\[0.3em]
	& & - \big\|\nabla\bt(\va)[\va]\big\|^2 \\[0.3em]
	&=&   \inp{\nabla \cj(\va)}{\va}-\inp{ \nabla^2 \Psi(z_\va)[z_\va+w_\va]}{z_\va+w_\va} + \inp{\nabla\Psi(z_\va)}{z_\va}\\[0.3em]
	& & - \inp{\nabla\Psi(z_\va)}{\bt(\va)} - \big\|\nabla\bt(\va)[\va]\big\|^2 \\[0.3em]
	&\overset{\eqref{reduction equ}}{=}&  \inp{\nabla \cj(\va)}{\va}-\inp{ \nabla^2 \Psi(z_\va)[z_\va+w_\va]}{z_\va+w_\va} + \inp{\nabla\Psi(z_\va)}{z_\va}\\[0.3em]
	& & + \|\bt(\va)\|^2 - \big\|\nabla\bt(\va)[\va]\big\|^2 \\[0.3em]
	&\overset{\eqref{reduction equ}}{=}&\inp{\nabla \cj(\va)}{\va} - \inp{ \nabla^2 \Psi(z_\va)[z_\va+w_\va]}{z_\va+w_\va} + 2\inp{\nabla\Psi(z_\va)}{w_\va}\\[0.3em]
	& & + \inp{\nabla\Psi(z_\va)}{z_\va}-\|w_\va\|^2 \\[0.3em]
	& \overset{\text{(H5)}}{\leq} &\inp{\nabla \cj(\va)}{\va} -(\ka-1)\inp{\nabla\Psi(z_\va)}{z_\va}-\|w_\va\|^2
	%\endaligned
	\end{eqnarray*}
	for $\va\in X$. Hence the assertion follows.
\end{proof}

It follows from the above two  lemmas and the second part of (H2) that, for $\va\in\msn$, $\nabla \ck(\va)\neq 0$. Hence $\msn$ is locally a $C^1$-manifold. Moreover, if $\va\in\msn$ is a constrained critical point of $\cj$, then $\va$ is also a critical point of $\cj$ on $X$. Indeed, according to the Lagrange multiplier rule, we have the existence of $\lm\in\R$ such that $\nabla \cj(\va)=\lm\nabla \ck(\va)$. Note that $\va\in\msn$, we find $\inp{\nabla \cj(\va)}{\va}=0$. Together with Lemma \ref{K-al estimate}, this implies $\lm=0$ and $\va$ is a critical point of $\cj$ on $X$.

\begin{Rem}\label{rem:uniquenss-Nehari}
For later use, we also mention that if $\msn\neq\emptyset$ then $\msn$ is bounded away from $0$, i.e., there exists $r_0>0$ such that $\|\va\|\geq r_0$ for all $\va\in\msn$. In fact, for a given $\va\in S^X$, let us consider the function $t\mapsto \cj(t\va)$, $t\geq0$. Then we have
\[
\frac{d}{dt}\cj(t\va)=\inp{\nabla \cj(t\va)}{\va}=\frac{\ck(t\va)}t 
\]
and 
\[
\frac{d^2}{dt^2}\cj(t\va)=\inp{\nabla^2 \cj(t\va)[\va]}{\va}=\frac1{t^2}\big( \inp{\nabla \ck(t\va)}{t\va}-\ck(t\va) \big).
\]
It follows from Lemmas \ref{lemma 1}, \ref{K-al estimate} and the second part of (H2) that if there is $t_\va>0$ such that $\ck(t_\va\va)=0$, i.e., $t_\va\va\in\msn$, then $\frac{d^2}{dt^2}\cj(t\va)\big|_{t=t_\va}<0$. And hence $t_\va$ is a strict local maximum point. We actually have more. In fact, by Lemmas \ref{lemma 1} and \ref{K-al estimate}, we can see that $\ck(t\va)<0$ for all $t>t_\va$. Therefore, $t_\va$ should be the unique critical point of $\cj(t\va)$, $t>0$, and hence the global maximum point. Now, by \eqref{linking1}, we find $\max_{t>0}\cj(t\va)\geq\tau$, and this implies the existence of $r_0>0$ such that $t_\va\geq r_0$.
\end{Rem}

We also note that the min-max value in \eqref{ga-al1} has the following equivalent characterizations
\begin{\equ}\label{ga-al2}
	\ga=\inf_{\va\in X\setminus\{0\}}\max_{t>0}\cj(t\va)=\inf_{\va\in X\setminus\{0\}}\max_{z\in\R^+\va\op Y}\cl(z)
	=\inf_{\va\in\msn}\cj(\va).
\end{\equ}
Particularly,  $\ga<+\infty$ provided that $\msn\neq\emptyset$.

Next, let us consider a sequence of ``almost critical points" $\{z_n\}\subset\ch$ such that 
\begin{\equ}\label{key assumption}
	c_1\leq \cl(z_n)\leq c_2 \quad \text{and} \quad 
	\|\nabla\cl(z_n)\|\to0 \quad \text{as } n\to\infty,
\end{\equ}
for some fixed constants $c_1,c_2>0$. In this setting, $\{z_n\}$ is nothing but a generic Palais-Smale sequence for $\cl$. 

\begin{Lem}\label{lemma 3}
	Under the condition \eqref{key assumption}, we have:
	\begin{itemize}
		\item[$(1)$] $\|z_n\|$ is uniformly bounded in $n$;
		
		\item[$(2)$] $\|z_n^Y-\bt(z_n^X)\|\leq \|\nabla\cl(z_n)\|$ as $n\to\infty$;
		
		\item[$(3)$] $\nabla \cj(z_n^X)\to0$ as $n\to\infty$.
	\end{itemize}
\end{Lem}
\begin{proof}
	We can derive from \eqref{key assumption} and (H2) that
	\[
	\aligned
	\inp{\nabla\Psi(z_n)}{z_n}&=\|z_n^X\|^2-\|z_n^Y\|^2-\inp{\nabla\cl(z_n)}{z_n}   \\[0.3em]
	&=2\cl(z_n)+2\Psi(z_n)+o_n(1)\|z_n\| \\
	&\leq 2c_2 +\frac2p\inp{\nabla\Psi(z_n)}{z_n}+o_n(1)\|z_n\|.
	\endaligned
	\]
	Hence
	\[
	\inp{\nabla\Psi(z_n)}{z_n}\leq \frac p{p-2}\big( 2c_2+o_n(1)\|z_n\| \big)
	\]
	and so, by (H4),
	\[
	\aligned
	\|z_n\|&= \|z_n^X-z_n^Y\| \leq \|\nabla\cl(z_n)\|+\|\nabla\Psi(z_n)\| \\[0.3em]
	&\leq o_n(1)+ K\inp{\nabla\Psi(z_n)}{z_n}^\mu  +\mu\|z_n\| \\
	&\leq o_n(1)+ K \frac{p^\mu}{(p-2)^\mu}\big( 2c_2+o_n(1)\|z_n\| \big)^\mu+\mu\|z_n\| .
	\endaligned
	\]
	Since $\mu<1$, this implies the boundedness of $\|z_n\|$.
	
	To check $(2)$, let us set $\phi_n=z_n^X+\bt(z_n^X)$ and $w_n=z_n-\phi_n$. Then, we have $w_n=z_n^Y-\bt(z_n^X)\in Y$. Because of the relation \eqref{reduction equ} and the fact $\bt(z_n^X)$ is the maximum of $\cl_{z_n^X}$, we see that
	\[
	0=\inp{\nabla\cl(\phi_n)}{w_n}=-\inp{\phi_n^Y}{w_n}-\inp{\nabla\Psi(\phi_n)}{w_n}.
	\]
	Since 
	\[
	\inp{\nabla\cl(z_n)}{w_n}=-\inp{z_n^Y}{w_n}-\inp{\nabla\Psi(z_n)}{w_n},
	\]
	it follows that
	\begin{\equ}\label{X2}
		\inp{\nabla\cl(z_n)}{w_n}=-\|w_n\|^2+\inp{\nabla\Psi(\phi_n)}{w_n}-\inp{\nabla\Psi(z_n)}{w_n}
	\end{\equ}
	where the last two summands add up to something non-positive. Indeed, by the convexity in (H3), we can see that
\[
\inp{\nabla\Psi(z_n)}{w_n}-\inp{\nabla\Psi(\phi_n)}{w_n}=\inp{\nabla^2 \Psi(\phi_n+\theta_n w_n)[w_n]}{w_n}\geq0
\]
for some $\theta_n\in(0,1)$. Together with  \eqref{X2}, this implies
\[
\|w_n\|^2\leq \|\nabla\cl(z_n)\|\cdot\|w_n\|,
\]
and hence $\|w_n\|\leq \|\nabla\cl(z_n)\|$ as $n\to\infty$.

To see $(3)$, it follows from $(2)$ and the $C^2$-smoothness of $\cl$ that
\[
\|\nabla\cl(\phi_n)-\nabla\cl(z_n)\|\leq \|\nabla^2\cl(\phi_n+\tilde\theta_nw_n)[w_n]\|=O(\|w_n\|)
\] 
where $\tilde\theta_n\in(0,1)$. Here we have used the uniform boundedness of $\nabla^2\cl(\phi_n+\tilde\theta_nw_n)$. And thus we get
$\|\nabla \cl(\phi_n)\|\to0$ as $n\to\infty$. Finally, since $\|\nabla \cj(z_n^X)\|=\|\nabla\cl(\phi_n)\|$, we have that $\|\nabla \cj(z_n^X)\|\to0$ as $n\to\infty$.
\end{proof}

The next lemma guarantees that the infimum in \eqref{ga-al2} is well-defined, i.e.,  $\ga<\infty$.
\begin{Lem}\label{lemma 4}
	Under the condition \eqref{key assumption}, we have $\msn\neq \emptyset$. In particular, there exists $t_n>0$ such that $t_nz_n^X\in\msn$ and $|t_n-1|=O(\|\nabla \cj(z_n^X)\|)$ as $n\to\infty$.
\end{Lem}
\begin{proof}
	Let's use the same notation $\phi_n=z_n^X+\bt(z_n^X)$ as in the previous lemma. Then, from \eqref{key assumption} and Lemma \ref{lemma 3} $(3)$, it follows that
	\begin{\equ}\label{X3}
		\liminf_{n\to\infty}\inp{\nabla\Psi(\phi_n)}{\phi_n}\geq c_0
	\end{\equ}
	for some constant $c_0>0$. In fact, we have $\inp{\nabla\Psi(\phi_n)}{\phi_n}>0$ from (H2). Furthermore, if  $\inp{\nabla\Psi(\phi_n)}{\phi_n}\to 0$ as $n\to\infty$ then 
	\[
	\aligned
	\|\phi_n\|&=\|\phi_n^X-\phi_n^Y\|\leq \|\nabla \cl(\phi_n)\| + \|\nabla\Psi(\phi_n)\| \\[0.3em]
	&\leq \|\nabla \cl(\phi_n)\| + K\inp{\nabla\Psi(\phi_n)}{\phi_n}^\mu+\mu\|\phi_n\| \quad \text{by (H4)} \\[0.3em]
	&=o_n(1)+\mu\|\phi_n\|.
	\endaligned
	\]
	Since $\mu<1$, we find $\|\phi_n\|\to0$ as $n\to\infty$. This contradicts to $\cl(\phi_n)\geq \cl(z_n)\geq c_1>0$ as assumed in \eqref{key assumption}.
	
	Let us set $\eta_n:[0,\infty)\to\R$ by $\eta_n(t)=\ck(tz_n^X)$, where $\ck$ is given by \eqref{K-al def}. One easily checks that $\eta_n'(t)=\inp{\nabla \ck(tz_n^X)}{z_n^X}$. And hence, by Lemma \ref{K-al estimate}, the Taylor's expansion of $\eta_n$ around $t=1$ and the continuity of $\nabla^2\Psi$ and $\nabla\bt$, we get
	\begin{\equ}\label{X4}
	\aligned
		t\eta_n'(t)&\leq 2\eta_n(t)-(\ka-1)\inp{\nabla\Psi(tz_n^X+\bt(tz_n^X))}{tz_n^X+\bt(tz_n^X)} \\[0.3em]
		&\leq 2\eta_n(1) - (\ka-1)\inp{\nabla\Psi(\phi_n)}{\phi_n} + O(|t-1|)
	\endaligned
	\end{\equ}
	where the $O(|t-1|)$ term is independent of $n$ due to the uniform boundedness of $\eta_n'(t)$ in both $n$ and $t$ in a neighborhood of $1$ (this is a direct consequence of Lemma \ref{lemma 3} $(1)$).
	
	Note that $\eta_n(1)=\inp{\nabla \cj(z_n^X)}{z_n^X}\to0$ (by Lemma \ref{lemma 3} $(3)$), we can conclude from \eqref{X3} and \eqref{X4} that there exists a small constant $\nu>0$ such that
	\[
	\eta_n'(t)\leq -\nu \text{ for all } t\in[1-\nu,1+\nu] \text{ and } n \text{ large enough.}
	\] 
	In particular, $\eta_n(1-\nu)\geq\eta_n(1)+\nu^2>0$ and $\eta_n(1+\nu)\leq \eta_n(1)-\nu^2<0$ for $n$ large. Then, by the Inverse Function Theorem, $t_n=\eta_n^{-1}(0)$ exists and $t_nz_n^X\in\msn$ for $n$ large enough. Furthermore, since $|\eta_n'(t)^{-1}|\leq \frac1\nu$ on $[1-\nu,1+\nu]$, we consequently get
	\[
	|t_n-1|=\big|\eta_n^{-1}(0)-\eta_n^{-1}(\ck(z_n^X))\big|\leq \frac1\nu |\ck(z_n^X)|.
	\]
	Now the conclusion follows from $\ck(z_n^X)=O(\|\nabla \cj(z_n^X)\|)$.
\end{proof}

Combining Lemma \ref{lemma 1} to \ref{lemma 4}, we have the following general result, providing us an upper bound estimate for the minimax value in \eqref{ga-al2}.

\begin{Thm}\label{abstract thm}
	Under the hypotheses {\rm (H1)} to {\rm (H5)}. If there exists $\{z_n\}\subset\ch$ verifying \eqref{key assumption}, then the min-max value $\ga>0$ defined in \eqref{ga-al2} satisfies
	\[
	\ga\leq \cl(z_n)+O(\|\nabla\cl(z_n)\|^2) \quad \text{as } n\to\infty.
	\]
	If, additionally, the functional $\cl$ satisfies the Palais-Smale condition at the energy level $\ga$, then there exists $z\in\ch\setminus\{0\}$ such that $\cl(z)=\ga$ and $\nabla\cl(z)=0$.
\end{Thm}

\begin{Rem}
	It is easy to prove that, under the hypotheses of Theorem \ref{abstract thm}, there is a natural upper bound of $\ga$, that is the constant $c_2$ in the assumption \eqref{key assumption}. However, this is not enough since sometimes it is crucial to have a more explicit asymptotic characterization for the energies (see for instance the resolution of classical Yamabe problem and its variations, elliptic Brezis-Nirenberg problem, etc.). As we will see later in the applications, the sequence $\{z_n\}$ in Theorem \ref{abstract thm} plays the  role of certain test elements for the functional  $\cl$. Our result shows a refined upper bound estimate for the min-max value $\ga$, and  makes it computable whenever $\{z_n\}$ is explicitly constructed. 
\end{Rem}

\begin{proof}[Proof of Theorem \ref{abstract thm}]
	
	We set $\phi_n=z_n^X+\bt(z_n^X)$ as before and $t_n>0$ as was in Lemma \ref{lemma 4}, we also let $\psi_n=t_nz_n^X+\bt(t_nz_n^X)$. It follows from Lemma \ref{lemma 3} and \ref{lemma 4} that
	\begin{\equ}\label{X5}
		\aligned
		\|z_n-\psi_n\|&\leq \|z_n-\phi_n\|+|t_n-1|\|z_n^X\|  +\|\bt(z_n^X)-\bt(t_nz_n^X)\| \\[0.3em]
		&=O(\|\nabla\cl(z_n)\|)+O(\|\nabla \cj(z_n^X)\|)
		\endaligned
	\end{\equ}
	where we have used the inequality
	\[
	\|\bt(z_n^X)-\bt(t_nz_n^X)\| \leq \normm{\nabla\bt(\theta_nz_n^X)}\cdot\|z_n^X\|\cdot|t_n-1|
	\]
	for some $\theta_n$ between $t_n$ and $1$. Since we have $\|\nabla \cj(z_n^X)\|=\|\nabla\cl(\phi_n)\|$ (by Proposition \ref{reduction} $(2)$), we can see from Lemma \ref{lemma 3} that
	\[
	\|\nabla \cj(z_n^X)\|\leq
	\|\nabla\cl(z_n)\| + O(\|\phi_n-z_n\|)=O(\|\nabla\cl(z_n)\|).
	\]
	Together with \eqref{X5}, this implies 
	$
	\|z_n-\psi_n\|=O(\|\nabla\cl(z_n)\|)
	$.
	
	Now, using Taylor's expansion, we obtain
	\[
	\aligned
	\cl(z_n)&=\cl(\psi_n)+\inp{\nabla\cl(\psi_n)}{z_n-\psi_n} + O(\|z_n-\psi_n\|^2) \\[0.3em]
	&=\cl(\psi_n)+\inp{\nabla\cl(\psi_n)}{z_n-\psi_n} + O(\|\nabla\cl(z_n)\|^2).
	\endaligned
	\]
	Since $\psi_n=t_nz_n^X+\bt(t_nz_n^X)$ and $t_nz_n^X\in\msn$, we have (by \eqref{reduction equ})
	\[
	\inp{\nabla\cl(\psi_n)}{z_n-\psi_n}=\inp{\nabla\cl(\psi_n)}{(z_n-\psi_n)^X}=(1-t_n)\inp{\nabla \cj(t_nz_n^X)}{z_n^X}=0.
	\]
	Therefore, we have
	\[
	\ga=\inf_{\va\in\msn}\cj(\va)\leq \cj(t_nz_n^X)=\cl(\psi_n)=\cl(z_n)+ O(\|\nabla\cl(z_n)\|^2)
	\]
	as desired.
	
	Finally, let us consider a minimizing sequence $\{\va_j\}\subset\msn$ for the reduced functional $\cj$. By Proposition \ref{reduction}, we see that $\{z_j=\va_j+\bt(\va_j)\}\subset\ch$ is a $(PS)_\ga$-sequence for $\cl$. And hence the conclusion follows from the assumption that $\cl$ satisfies the $(PS)_\ga$-condition, which means that $\{z_j\}$ is compact. 
\end{proof}

\section{Geometric preliminaries}\label{sec: geometric preliminaries}

In this section, we still assume $(M,\ig,\sa)$ is a general closed Riemannian spin manifold. We will recall how to develop the Dirac operator in normal coordinates using a well-suited trivialization of the spinor bundle. Our presentation follows \cite{AGHM} where a similar development was calculated, but we present here a more accurate expression than that of \cite{AGHM} to lay the groundwork for the subsequent proof of Theorem \ref{main thm}.

\subsection{Geometric notations}

The curvature operator is given by 
\[
R(X,Y)=[\nabla_X,\nabla_Y]-\nabla_{[X,Y]}.
\]
Then the Riemannian curvature tensor is the tensor with components $R_{ijkl}$, evaluated in local coordinates $\{x^i\}$ by
\[
R_{ijkl}=\ig(R(\pa_k,\pa_l)\pa_j,\pa_i).
\]

From now on, we use Einstein's summation convention.
The Ricci tensor is given by the contraction $R_{ij}=\ig^{kl}R_{kilj}$ of the curvature tensor, and the scalar curvature is the trace $S=\ig^{ij}R_{ij}$ of the Ricci tensor.

The Weyl tensor is given by
\[
W_{ijkl}=R_{ijkl}-\frac1{m-2}(R_{ik}\ig_{jl}-R_{il}\ig_{jk}+R_{jl}\ig_{ik}-R_{jk}\ig_{il})+\frac{S}{(m-1)(m-2)}(\ig_{ik}\ig_{jl}-\ig_{il}\ig_{jk}).
\]

\subsection{The Bourguignon-Gauduchon trivialization}\label{B-G-T}

To begin with we fix $\Psi_0\in\mbs_m$ with $|\Psi_0| = 1$ arbitrarily and define 
\begin{\equ}\label{t0}
\psi(x)=\frac{m^{\frac{m-1}2}}{(1+|x|^2)^{\frac m2}}(1-x)\cdot_{\ig_{\R^m}}\Psi_0
\end{\equ}
for $x\in\R^m$ and $\cdot_{\ig_{\R^m}}$ denotes the Clifford multiplication with respect to the Euclidean metric. Let $D_{\R^m}$ be the Dirac operator on $\R^m$, then it is standard to verify that
\begin{\equ}\label{t1}
	D_{\R^m}\psi=|\psi|^{2^*-2}\psi
\end{\equ}
and
\begin{\equ}\label{t2}
	|\psi| = \left(\frac{m}{1+|x|^2}\right)^{\frac{m-1}{2}}.
\end{\equ}
We choose $\de<i(M)/2$ where $i(M)>0$ is the injectivity radius of $M$. Let $\eta:\R^m\to\R$ be a smooth radially symmetric cut-off function satisfying $\eta(x)=1$ if $|x|\le\de$ and $\eta(x)=0$ if $|x|\ge2\de$. Now we define  $\va_\vr\in C^\infty(\R^m,\mbs_m)$ by
\begin{\equ}\label{test spinor}
	\va_\vr(x)=\eta(x)\psi_\vr(x) \quad \text{where}\quad \psi_\vr(x)=\vr^{-\frac{m-1}2}\psi(x/\vr).
\end{\equ}

In order to transplant the test spinor on $M$, we employ the well adapted Bourguignon-Gauduchon-trivialization \cite{BG}. Here we fix $p_0\in M$ arbitrarily, and let $(x_1,\dots,x_m)$ be the normal coordinates given by the exponential map
\[
\exp_{p_0}: \R^m\cong T_{p_0}M\supset U \to V\subset M,\quad x \mapsto p = \exp_{p_0}(x).
\]
For $p\in V$ let $G(p)=(\ig_{ij}(p))_{ij}$ denote the corresponding metric at $p$. Since $G(p)$ is symmetric and positive definite, the square root
$B(p)=(b_{ij}(p))_{ij}$ of $G(p)^{-1}$ is well defined, symmetric and positive definite. It can be thought of as a linear isometry
\[
B(p): (\R^m\cong T_{\exp_{p_0}^{-1}(p)}U,\ig_{\R^m}) \to (T_pV,\ig).
\]
We obtain an isomorphism of $SO(m)$-principal bundles:
\begin{displaymath}
	\xymatrix{
		P_{SO}(U,\ig_{\R^m}) \ar[r]^{\displaystyle\phi}  \ar[d] & P_{SO}(V,\ig) \ar[d] \\
		T_{p_0}M \supset U\ar[r]^{\ \ \displaystyle\exp_{p_0}} & V \subset M}
\end{displaymath}
where $\phi(y_1,\dots,y_m) = (By_1,\dots,By_m)$ for an oriented frame $(y_1,\dots,y_m)$ on $U$. Notice that $\phi$ commutes with the right action of $SO(m)$, hence it induces an isomorphism of spin structures:
\begin{displaymath}
	\xymatrix{
		Spin(m)\times U  =P_{Spin}(U,\ig_{\R^m}) \ar[r] \ar[d] & P_{Spin}(V,\ig) \subset P_{Spin}(M) \ar[d]\\
		T_{p_0}M\supset U\ar[r]^{\ \ \displaystyle\exp_{p_0}} & V \subset M}
\end{displaymath}
Thus we obtain an isomorphism between the spinor bundles $\mbs(U)$ and $\mbs(V)$:
\begin{equation}\label{spin-iso}
	\mbs(U) := P_{Spin}(U,\ig_{\R^m})\times_\rho \mbs_m \longrightarrow \mbs(V) := P_{Spin}(V,\ig)\times_\rho \mbs_m \subset \mbs(M)
\end{equation}
where $(\rho,\mbs_m)$ is the complex spin representation.

Setting $e_i=B(\pa_i)=\sum_{j}b_{ij}\pa_j$ we obtain an orthonormal frame $(e_1,\dots, e_m)$ of $(TV,\ig)$. In order to simplify the notation, we use $\nabla$ and $\bar\nabla$, respectively, for the Levi-Civita connections on $(TU,\ig_{\R^m})$ and $(TV,\ig)$ and for the natural lifts of these connections to the spinor bundles $\mbs(U)$ and $\mbs(V)$, respectively. For the Clifford multiplications on these bundles, we shall simply write ``$\cdot$" in both cases, that is,
\[
e_i\cdot\bar\psi=B(\pa_i)\cdot\bar\psi=\ov{\pa_i\cdot\psi}.
\]

Now a spinor $\va\in\Ga(\mbs(U))$ corresponds via the isomorphims \eqref{spin-iso} to a spinor $\bar\va\in\Ga(\mbs(V))$. In particular, since the spinors $\va_\vr\in\Ga(\mbs(U))$ from \eqref{test spinor} have compact support in $U$ they correspond to spinors $\bar\va_\vr\in\Ga(\mbs(M))$ with compact support in $V$. These will play the role of the test spinors in our functional settings. In the rest of this subsection we give some calculations that will be needed in order to establish the energy estimate later.

We write $D$ and $\bar D$ for the Dirac operators acting on $\Ga(\mbs(U))$ and $\Ga(\mbs(V))$, respectively. There holds (see \cite{AGHM} for details)
\begin{\equ}\label{cut-off spinor identity}
	\bar D \bar\va_\vr = \ov{D\va_\vr}+\Theta\cdot\bar\va_\vr + \La\cdot\bar\va_\vr + \sum_{i,j}(b_{ij}-\de_{ij})\ov{\pa_i\cdot\nabla_{\pa_j}\va_\vr}
\end{\equ}
with $\Theta\in\Ga(Cl(TV))$ and $\La\in\Ga(TV)$ given by
\[
\Theta = \frac14\sum_{\substack{i,j,k \\ i\neq j\neq k\neq i}}\sum_{\al,\bt} b_{i\al}(\pa_{\al}b_{j\bt})(b^{-1})_{\bt k}e_i\cdot e_j\cdot e_k,
\]
and
\[
\La = \frac14\sum_{i,k}\big( \bar\Ga_{ik}^i-\bar\Ga_{ii}^k\big)e_k = \frac12\sum_{i,k} \bar\Ga_{ik}^i e_k;
\]
here $(b^{-1})_{ij}$ denotes entries of the inverse matrix of $B$, and $\bar\Ga_{ij}^k:=\ig(\bar\nabla_{e_j}e_j,e_k)$.

\subsection{Expansion of the metric}

For any point $p_0\in M$, $r$ denotes the distance from $p$ to $p_0$. Recall that in the neighborhood of $p_0$, we have
\begin{\equ}\label{development of the metric}
	\aligned
	\ig_{ij}&=\de_{ij}+\frac13R_{i\al\bt j}x^\al x^\bt+\frac16R_{i\al\bt j,k}x^\al x^\bt x^k \\
	&\qquad +\Big( \frac1{20} R_{i\al\bt j,kl}+\frac2{45}\sum_d R_{i\al\bt d}R_{jkld}\Big)x^\al x^\bt x^k x^l + O(r^5)
	\endaligned
\end{\equ}
where we used the Einstein's summation convention.
Therefore, we can write
\[
G=I+G_2+G_3+G_4+O(r^5)
\]
with
\[
(G_2)_{ij}=\frac13R_{i\al\bt j}x^\al x^\bt
\]
\[
(G_3)_{ij}=\frac16R_{i\al\bt j,k}x^\al x^\bt x^k
\]
and 
\[
(G_4)_{ij}=\Big( \frac1{20} R_{i\al\bt j,kl}+\frac2{45}\sum_d R_{i\al\bt d}R_{jkld}\Big)x^\al x^\bt x^k x^l
\]
Let $B$ be the matrix in the Bourguignon-Gauduchon trivialization, then we have
\[
B=I+B_1+B_2+B_3+B_4+O(r^5).
\]
The relation $B^2G=I$ yields $B_1=0$ and
\[
2B_2+G_2=2B_3+G_3=2B_2G_2+2B_4+B_2^2+G_4=0.
\]
Hence
\[
\aligned
b_{ij}&=\de_{ij}-\frac16R_{i\al\bt j}x^\al x^\bt-\frac1{12}R_{i\al\bt j,k}x^\al x^\bt x^k \\
&\qquad - \Big( \frac1{40} R_{i\al\bt j,kl}-\frac7{360}\sum_d R_{i\al\bt d}R_{jkld}\Big)x^\al x^\bt x^k x^l +O(r^5)
\endaligned
\]
and similarly for the inverse of $B$ we have
\[
\aligned
(b^{-1})_{ij}&=\de_{ij}+\frac16R_{i\al\bt j}x^\al x^\bt+\frac1{12}R_{i\al\bt j,k}x^\al x^\bt x^k \\
&\qquad + \Big( \frac1{40} R_{i\al\bt j,kl}+\frac1{120}\sum_d R_{i\al\bt d}R_{jkld}\Big)x^\al x^\bt x^k x^l +O(r^5)
\endaligned
\]

Likewise, we have
\[
\Theta=-\frac1{144}\sum_{\substack{i,j,k \\ i\neq j\neq k\neq i}}\sum_l R_{l\bt\ga k}(R_{ji\al l}+R_{jl\al i})x^\al x^\bt x^\ga e_i\cdot e_j \cdot e_k +O(r^4)
\]
and 
\[
\La=-\sum_k\Big( \frac14 R_{\al k}x^\al +\frac16 R_{\al k,\bt}x^\al x^\bt + O(r^3)\Big)e_k
\]
in the expansion \eqref{cut-off spinor identity}.
%These two expansions can also be found in \cite{AGHM}. 

\section{Proof of the main results}\label{sec: proof the main results}

From now on, let us consider $(M,\ig)$ is not locally conformally flat. The proof is divided into three major steps: 

\begin{itemize}
	\item[1.]  The construction of a good test spinor by using a Killing spinor on the $m$-sphere and some fundamental estimates, which are derived from elementary computations by using the very test spinor. 
	
	\item[2.] The proof of Theorem \ref{main thm} for the case $\ker(D_\ig)=\{0\}$.
	
	\item[3.] The proof of Theorem \ref{main thm} for the case $\ker(D_\ig)\neq\{0\}$.
\end{itemize}
Let us remind the reader that the proof of Theorem \ref{main thm} will be accomplished by applying the local energy estimate in Theorem \ref{abstract thm} to the functional $\Phi$ given in \eqref{the functional} (see also its simplified formulation \eqref{the functional in E}).

\subsection{Some estimates}

For $\dim M=m\geq4$, the local conformal flatness of $(M,\ig)$ is characterized by the nullity of the Weyl tensor $W_\ig$. Since we assumed that $(M,\ig)$ is not conformally flat, there exists $p_0\in M$ such that $|W_\ig(p_0)|>0$. In the sequel, we fix this $p_0$. Up to a conformal change of the metric, let us consider the following conformal normal coordinates, introduced by Lee and Parker \cite[Section 5]{LeeParker}. By conformal invariance of the Weyl tensor, the relation $|W_\ig(p_0)|>0$ still holds.

\begin{Prop}\label{conformal normal coordinates}
	{Given $\tau\in\R$, there is a conformal metric $\ig$ on $M$ such that}
	\begin{\equ}\label{conformal-normal-det}
{	\det \ig_{ij}=1+\tau\sum_{i,d}R_{i\al\bt d}R_{ikld}x^\al x^\bt x^k x^l+O(r^5), }
	\end{\equ}
	where $r=|x|$ in $\ig$-normal coordinates at $p_0$. Moreover, there hold
	\begin{itemize}
		\item[$(1)$] $R_{ij}=0$;
		
		\item[$(2)$] $R_{ij,k}+R_{jk,i}+R_{ki,j}=0$;
		
		\item[$(3)$] $\displaystyle \Big(R_{\al\bt ,kl}+\Big(\frac{2}9+20\tau\Big)\sum_{i,d} R_{i\al\bt d}R_{ikld}\Big)x^\al x^\bt x^k x^l=0$.
	\end{itemize}
\end{Prop}
\begin{proof}
{Since the idea is built upon the arguments in \cite[Section 5]{LeeParker}, we will simply sketch the proofs as follows. For a given metric $\ig$ on $M$, let $\{x^i\}$ denote the $\ig$-normal coordinates in a neighborhood $U$ of $p_0$. Then, using \eqref{development of the metric}, one should be able to obtain the Taylor expansion of $\det \ig_{ij}$ evaluated at $p_0$ as
\[
\aligned
\det\ig_{ij}&=1-\frac13 R_{\al\bt}x^\al x^\bt -\frac16 R_{\al\bt ,k}x^\al x^\bt x^k  \\
&\qquad - \Big( \frac1{20}R_{\al\bt ,kl}+\frac{1}{90}\sum_{i,d} R_{i\al\bt d}R_{ikld}-\frac1{18}R_{\al\bt}R_{kl} \Big)x^\al x^\bt x^k x^l+O(r^5),
\endaligned
\]
which can be found in \cite[Lemma 5.5]{LeeParker}.}

{Let $\msp_n$ denote the space of homogeneous polynomials in $x$ of degree $n$. For a function $f=f_2+f_3+f_4$ with $f_2\in\msp_2$, $f_3\in\msp_3$ and $f_4\in\msp_4$, let us consider the conformal metric $\tilde\ig=e^{2f}\ig$ and the $\tilde\ig$-normal coordinates $\{\tilde x^i=e^{-f}x^i\}$. In the sequel, quantities marked with a tilde will be associated with $\tilde\ig$, while those unmarked with such will be associated with $\ig$.}

{It follows directly that $\det \tilde\ig_{ij}=e^{2mf}\det \ig_{ij}$, and then by Taylor's theorem
\[
\aligned
\det\tilde\ig_{ij}&=1+\Big(-\frac13 R_{\al\bt}x^\al x^\bt + 2mf_2\Big)+\Big(-\frac16 R_{\al\bt ,k}x^\al x^\bt x^k + 2m f_3\Big) \\[0.5em]
&\qquad + \Big[ - \Big( \frac1{20}R_{\al\bt ,kl}+\frac{1}{90}\sum_{i,d} R_{i\al\bt d}R_{ikld}-\frac1{18}R_{\al\bt}R_{kl} \Big)x^\al x^\bt x^k x^l  \\
&\qquad \quad +2mf_4-\frac23(m-1)f_2R_{\al\bt}x^\al x^\bt  + 2m^2 f_2^2 \Big] + O(r^5).
\endaligned
\]
Here we have used the fact that $\tilde x^i=(1-f_2)x^i + O(r^4)$. Now, it is obvious that $f_2$ and $f_3$ are uniquely determined if $\det\tilde\ig_{ij}$ vanishes to order $4$ in $\tilde\ig$-normal coordinates.}

{To determine $f_4$, let us denote $\tilde\ig^*=e^{2(f_2+f_3)}\ig$ and $\{\tilde R^*_{ijkl}\}$ the components of the $\tilde\ig^*$-curvature tensor. Then we have $\tilde\ig=e^{2f_4}\tilde\ig^*$. Notice that $\det\tilde\ig^*_{ij}$ also vanishes to order $4$, we have 
\[
\tilde R^*_{ij}=0 \quad \text{and} \quad \tilde R^*_{ij,k}+\tilde R^*_{jk,i}+\tilde R^*_{ki,j}=0.
\]
Moreover, since $f_4\in\msp_4$, we have that the conformal factor $e^{2f_4}$ keeps the $\tilde\ig^*$-curvature tensor and its first derivatives invariant at $p_0$, i.e., we have $\tilde R_{ijkl}=\tilde R^*_{ijkl}$ and $\tilde R_{ij,k}=\tilde R^*_{ij,k}$ (cf. \cite[Theorem 5.2]{LeeParker}). Therefore, by choosing $f_4$ so that
\[
- \Big( \frac1{20}\tilde R^*_{\al\bt ,kl}+\frac{1}{90}\sum_{i,d} \tilde R^*_{i\al\bt d}\tilde R^*_{ikld} \Big)x^\al x^\bt x^k x^l+2mf_4
=\tau\sum_{i,d}\tilde R^*_{i\al\bt d}\tilde R^*_{ikld}x^\al x^\bt x^k x^l,
\]
we obtain \eqref{conformal-normal-det} by replacing $\ig$ with $\tilde\ig$.}
\end{proof}

\begin{Rem}\label{rmk-conformal-normal-coordinates}
{It is worth noting that we have retained the 4th-order term and the parameter $\tau$ in expansion \eqref{conformal-normal-det}. The reason for not enforcing $\det{g_{ij}}=1+O(r^N)$ with $N\geq5$ as stated in \cite[Theorem 5.1]{LeeParker} (or simply $\det{g_{ij}}\equiv1$ as in \cite{Cao, Gunther}) is to demonstrate that the presence of $\sum_{i,d}R_{i\al\bt d}R_{ikld}x^\al x^\bt x^k x^l$ is not crucial in deriving the energy bounds for functional \eqref{the functional}. We believe this information will be valuable in our upcoming works concerning Dirac operators and spinor equations. For reader's convenience, we would like to emphasize that, as demonstrated in the proof of Lemma \ref{energy-esti} below, the specific value of the parameter $\tau$ in \eqref{conformal-normal-det} has no impact on the conclusion of Lemma \ref{energy-esti}. Hence, one can simply choose here $\tau=0$ if desired.
}
\end{Rem}

%Let us remark that the coefficient in the third equality in Proposition \ref{conformal normal coordinates} is slightly different form the one in (5.5)(c) of \cite[Page 61]{LeeParker}. This is because we found a missing term in the final step of the proof of \cite[Lemma 5.5]{LeeParker}, but it doesn’t affect the result for the Yamabe problem. Since it is very crucial in our argument, we give a corrected formulation and sketch the proof later in the Appendix.
%

In what follows, whenever we refer to conformal normal coordinates, it will cause no confusion if we take the metric $\ig$ just as in Proposition \ref{conformal normal coordinates}. Next we collect some useful lemmas that will help us to obtain an accurate energy estimate when applying Theorem \ref{abstract thm}.

\begin{Lem}\label{lem: zero identity1}
For $\psi$ be given in \eqref{t0}, we have
\[
\sum_{i,j} R_{i\al\bt j}x^\al x^\bt\pa_i\cdot\nabla_{\pa_j}\psi(x)\equiv0
\]
for all $x\in\R^m$.
\end{Lem}
\begin{proof}
	Using the explicit formula of $\psi$, we get
	\[
	\pa_i\cdot\nabla_{\pa_j}\psi(x)=-\frac{m^{\frac{m-1}2}}{(1+|x|^2)^{\frac m2+1}}x_j\pa_i\cdot(1-x)\cdot\Psi_0-\frac{m^{\frac{m-1}{2}}}{(1+|x|^2)^{\frac m2}}\pa_i\cdot\pa_j\cdot\Psi_0.
	\]
	Hence, we compute
	\[
	\aligned
	R_{i\al\bt j}x^\al x^\bt\pa_i\cdot\nabla_{\pa_j}\psi(x)&=-\frac{m^{\frac{m-1}2}}{(1+|x|^2)^{\frac m2+1}}R_{i\al\bt j}x^\al x^\bt x^j\pa_i\cdot(1-x)\cdot\Psi_0 \\
	&\qquad - \frac{m^{\frac{m-1}{2}}}{(1+|x|^2)^{\frac m2}}R_{i\al\bt j}x^\al x^\bt\pa_i\cdot\pa_j\cdot\Psi_0.
	\endaligned
	\]
	
	Since $R_{i\al\bt j}=-R_{i\al j\bt}$, we get that
	\[
	R_{i\al\bt j}x^\al x^\bt x^j\pa_i\cdot(1-x)\cdot\Psi_0=0.
	\]
	That is, the firs summand vanishes.
	
	For the second summand, since $\pa_i\cdot\pa_j+\pa_j\cdot\pa_i=-2\de_{ij}$ and since
	\[
	R_{i\al\bt j}x^\al x^\bt=R_{i\bt\al j}x^\al x^\bt=R_{j\al\bt i}x^\al x^\bt
	\]
	(we have used the Einstein's summation convention and the fact $R_{j\al\bt i}=R_{\bt ij\al}=R_{i\bt \al j}$), we can obtain from Proposition \ref{conformal normal coordinates}(1) that
	\[
	\sum_{i,j}R_{i\al\bt j}x^\al x^\bt\pa_i\cdot\pa_j\cdot\Psi_0=-\sum_i R_{i\al\bt i}x^\al x^\bt\Psi_0  \\
	= R_{\al\bt}x^\al x^\bt\Psi_0 =0.
	\]
	Hence the second summand vanishes, which proves the lemma.
	
\end{proof}

\begin{Lem}\label{lem: zero identity2}
Let $A_{ijkl}\in\R$ be constants, for $1\leq i,j,k,l\leq m$, then $\Psi_0$ in \eqref{t0} can be chosen such that 
\begin{\equ}\label{zero identity}
\sum_{\substack{i,j,k,l \\ i\neq j\neq k\neq i}} A_{ijkl}\real(\pa_i\cdot\pa_j\cdot\pa_k\cdot\pa_l\cdot\Psi_0,\Psi_0)=0.
\end{\equ}
\end{Lem}
\begin{proof}
	First of all, let $i\neq j\neq k\neq i$. It is easy to see that, for arbitrary $\Psi_0$ with unit length, $	(\pa_i\cdot\pa_j\cdot\pa_k\cdot\pa_l\cdot\Psi_0,\Psi_0)$ takes only real values for mutually distinct  $i,j,k,l$ and it takes only pure imaginary values for $l\in\{i,j,k\}$.
	
	For dimension $m=3$, since $1\leq l\leq 3$, we have
	\[
	\sum_{\substack{i,j,k,l \\ i\neq j\neq k\neq i}} A_{ijkl}\real(\pa_i\cdot\pa_j\cdot\pa_k\cdot\pa_l\cdot\Psi_0,\Psi_0)\equiv0
	\]
	for arbitrary choice of $\Psi_0$.
	
	For dimension $m=4$, the summation is taken over all $4!$ permutations of the form
	\[
	\real(\pa_i\cdot\pa_j\cdot\pa_k\cdot\pa_l\cdot\Psi_0,\Psi_0) \quad \text{with } i,j,k,l \text{ mutually distinct.}
	\]
	Hence, we only need to evaluate $\real(\pa_1\cdot\pa_2\cdot\pa_3\cdot\pa_4\cdot\Psi_0,\Psi_0)$. For this purpose, let us set $w\in Cl(R^m)$ by
	\[
	w^\pm=\frac{1\pm \pa_1\cdot\pa_2\cdot\pa_3\cdot\pa_4}2.
	\]
	Then, by using $\pa_1\cdot\pa_2\cdot\pa_3\cdot\pa_4\cdot w^\pm=w^\pm\cdot\pa_1\cdot\pa_2\cdot\pa_3\cdot\pa_4=\pm w^\pm$ and $w^+\cdot w^-=0$, we see that 
	\[
	\real(\pa_1\cdot\pa_2\cdot\pa_3\cdot\pa_4\cdot\Upsilon,\Upsilon)=|w^+\cdot\Upsilon|^2-|w^-\cdot\Upsilon|^2
	\]
	for a spinor $\Upsilon\in \mbs_m$. Hence we can choose suitably $\Upsilon=\Psi_0$ with $|\Psi_0|=1$ such that
	\[
	\real(\pa_1\cdot\pa_2\cdot\pa_3\cdot\pa_4\cdot\Psi_0,\Psi_0)=0,
	\]
	and this leads to the identity \eqref{zero identity}.

	For dimension $m=d\geq 5$, let us assume the conclusion holds true for dimension $d-1$. Notice that, for any spinor $\Upsilon\in\mbs_d$, we have
	\[
	\sum_{\substack{i,j,k,l \\ i\neq j\neq k\neq i}} A_{ijkl}\real(\pa_i\cdot\pa_j\cdot\pa_k\cdot\pa_l\cdot\Upsilon,\Upsilon)=I_1(\Upsilon)+I_2(\Upsilon)+I_3(\Upsilon)+I_4(\Upsilon)+I_5(\Upsilon)
	\]
	where
	\begin{eqnarray*}
	&& I_1(\Upsilon)=\sum_{\substack{i,j,k,l\in\{1,\dots,d-1\} \\ i\neq j\neq k\neq i} }  A_{ijkl}\real(\pa_i\cdot\pa_j\cdot\pa_k\cdot\pa_l\cdot\Upsilon,\Upsilon), \\
	&& I_2(\Upsilon)=\sum_{\substack{j,k,l \\d\neq j\neq k\neq d } }  A_{djkl}\real(\pa_d\cdot\pa_j\cdot\pa_k\cdot\pa_l\cdot\Upsilon,\Upsilon), \\
	&& I_3(\Upsilon)=\sum_{\substack{i,k,l \\i\neq d\neq k\neq i } }  A_{idkl}\real(\pa_i\cdot\pa_d\cdot\pa_k\cdot\pa_l\cdot\Upsilon,\Upsilon), \\
	&& I_4(\Upsilon)=\sum_{\substack{i,j,l \\i\neq j\neq d\neq i } }  A_{ijdl}\real(\pa_i\cdot\pa_j\cdot\pa_d\cdot\pa_l\cdot\Upsilon,\Upsilon), \\
	&& I_5(\Upsilon)=\sum_{\substack{i,j,k  \\i\neq j\neq k\neq i } }  A_{ijkd}\real(\pa_i\cdot\pa_j\cdot\pa_k\cdot\pa_d\cdot\Upsilon,\Upsilon) .
	\end{eqnarray*}
Since the summand $I_1(\Upsilon)$ can be treated similarly as in the dimension $d-1$, we first obtain a unit spinor $\Psi_1\in \mbs_{d}$ such that $I_1(\Psi_1)=0$. Then we have
	\[
	\aligned
	\sum_{\substack{i,j,k,l \\ i\neq j\neq k\neq i}} A_{ijkl}\real(\pa_i\cdot\pa_j\cdot\pa_k\cdot\pa_l\cdot\Psi_1,\Psi_1)=I_2(\Psi_1)+I_3(\Psi_1)+I_4(\Psi_1)+I_5(\Psi_1).
	\endaligned
	\]
If the remaining summands vanish identically, then we are done. If not, let us consider the unit spinor $\Psi_2=\pa_d\cdot\Psi_1$. Then we see that, for $i,j,k,l\in\{1,\dots,d-1\}$,
\[
\real(\pa_i\cdot\pa_j\cdot\pa_k\cdot\pa_l\cdot\Psi_2,\Psi_2)=\real(\pa_i\cdot\pa_j\cdot\pa_k\cdot\pa_l\cdot\Psi_1,\Psi_1) 
\]
and, if any one of $i,j,k,l$ equals to $d$,
\[
\real(\pa_i\cdot\pa_j\cdot\pa_k\cdot\pa_l\cdot\Psi_2,\Psi_2)=-\real(\pa_i\cdot\pa_j\cdot\pa_k\cdot\pa_l\cdot\Psi_1,\Psi_1).
\]
Therefore the function
\[
\Upsilon \mapsto \sum_{\substack{i,j,k,l \\ i\neq j\neq k\neq i}} A_{ijkl}\real(\pa_i\cdot\pa_j\cdot\pa_k\cdot\pa_l\cdot\Upsilon,\Upsilon)
\]
changes sign on the unit sphere of $\mbs_d$. And due to the continuity of this function, we get the existence of a unit spinor $\Upsilon=\Psi_0$ such that \eqref{zero identity} holds. The proof is hereby completed.
\end{proof}

In the sequel we use the notation $f_\vr\lesssim g_\vr$ for two functions $f_\vr$ and $g_\vr$, when there exists a constant $C>0$ independent of $\vr$ such that $f_\vr\leq C g_\vr$.

\begin{Lem}\label{derivative-esti}
	Let $\bar\va_\vr\in\mbs(V)$ be as in Subsection \ref{B-G-T} and set $\bar R_\vr:=\bar D\bar\va_\vr-|\bar\va_\vr|^{2^*-2}\bar\va_\vr$. Then
	\[
	\|\bar R_\vr\|_{E^*}
	\lesssim\begin{cases}
		\vr^{\frac{m-1}2} &\text{if } 2\leq m\leq 6, \\
		\vr^3|\ln\vr|^{\frac47} &\text{if } m=7, \\
		\vr^3 &\text{if } m\geq8,
	\end{cases}
	\]
	where $E^*$ stands for the dual space of $E=H^{\frac12}(M,\mbs(M))$.
\end{Lem}

\begin{proof}
In order to estimate $\bar R_\vr$,  we first deduce from \eqref{t1} and \eqref{t2} that
	\begin{\equ}\label{identity-D}
		\aligned
		D\va_\vr &= \nabla\eta\cdot\psi_\vr+\eta D\psi_\vr\\
		&=\nabla\eta\cdot\psi_\vr+|\va_\vr|^{2^*-2}\va_\vr+(\eta-\eta^{2^*-1})|\psi_\vr|^{2^*-2}\psi_\vr.
		\endaligned
	\end{\equ}
	Using \eqref{cut-off spinor identity}, we obtain
	\[
	\bar R_\vr=A_1+A_2+A_3+A_4+A_5+A_6
	\]
	where
	\begin{eqnarray*}
		&&A_1=\ov{\nabla\eta\cdot\psi_\vr},\\
		&&A_2=(\eta-\eta^{2^*-1})|\bar\psi_\vr|^{2^*-2}\bar\psi_\vr,\\
		&&A_3= \eta \Theta\cdot\bar\psi_\vr,\\
		&&A_4=\eta \La\cdot\bar\psi_\vr,\\
		&&A_5=\eta \sum_{i,j} (b_{ij}-\de_{ij})\ov{\pa_i\cdot\nabla_{\pa_j}\psi_\vr},\\
		&&A_6=\sum_{i,j}(b_{ij}-\de_{ij})\pa_j\eta\,\ov{\pa_i\cdot\psi_\vr}.
	\end{eqnarray*}
	
	In the following estimates we use that the support of $\eta$ is contained in $B_{2\de}(0)\subset\R^m$. For $\psi\in E$ with $\|\psi\|\leq1$ and $\vr$ small we estimate:\\
	
	%\begin{equation}\label{A1}
	$\displaystyle \begin{aligned}
		\|A_1\|_{E^*}
		&\lesssim\left(\int_{B_{2\de}(p_0)}\big|\ov{\nabla\eta\cdot\psi_\vr}\big|^{\frac{2m}{m+1}}d\vol_\ig \right)^{\frac{m+1}{2m}}
		\lesssim\left(\int_{\de\leq|x|\leq2\de}|\psi_\vr|^{\frac{2m}{m+1}}dx \right)^{\frac{m+1}{2m}}  \\[0.3em]
		&\lesssim\left( \vr^{\frac{2m}{m+1}}\int_{\frac\de\vr}^{\frac{2\de}\vr}\frac{r^{m-1}}{(1+r^2)^\frac{m(m-1)}{m+1}}dr \right)^{\frac{m+1}{2m}}
		\lesssim \ \vr^{\frac{m-1}2}
	\end{aligned}$\\
	%\end{equation}
	
	\medskip
	
	%\begin{equation}\label{A2}
	$\displaystyle \begin{aligned}
		\|A_2\|_{E^*}
		&\lesssim \left(\int_{B_{2\de}(p_0)}\big(\eta-\eta^{\frac{m+1}{m-1}}\big)^{\frac{2m}{m+1}}|\bar\psi_\vr|^{\frac{2m}{m-1}}d\vol_\ig\right)^{\frac{m+1}{2m}} \\[0.3em]
		&\lesssim \left(\int_{\de\leq|x|\leq2\de}|\psi_\vr|^{\frac{2m}{m-1}}dx \right)^{\frac{m+1}{2m}}
		\lesssim \left(\int_{\frac\de\vr}^{\frac{2\de}\vr}\frac{r^{m-1}}{(1+r^2)^m}dr \right)^{\frac{m+1}{2m}} \\
		&\lesssim \vr^{\frac{m+1}2}
	\end{aligned}$\\
	%\end{equation}
	
	\medskip
	
	%\begin{equation}\label{A3}
	$\displaystyle \begin{aligned}
		\|A_3\|_{E^*}
		&\lesssim \left(\int_{B_{2\de}(p_0)}|\Theta|^{\frac{2m}{m+1}}|\bar\psi_\vr|^{\frac{2m}{m+1}}d\vol_\ig \right)^{\frac{m+1}{2m}}
		\lesssim \left(\int_{|x|\leq2\de}|x|^{\frac{6m}{m+1}}|\psi_\vr|^{\frac{2m}{m+1}}dx \right)^{\frac{m+1}{2m}} \\[0.3em]
		&\lesssim \left( \vr^{\frac{8m}{m+1}}\int_0^{\frac{2\de}\vr}\frac{r^{\frac{6m}{m+1}+m-1}}{(1+r^2)^{\frac{m(m-1)}{m+1}}}dr \right)^{\frac{m+1}{2m}}
		\lesssim  \begin{cases}
			\vr^{\frac{m-1}2} &\text{if } 2\leq m\leq8\\
			\vr^4|\ln\vr|^{\frac59} &\text{if }  m=9\\
			\vr^4 &\text{if } m\geq10
		\end{cases}
	\end{aligned}$\\
	%\end{equation}
	
	\medskip
	
	%\begin{equation}\label{A4}
	$\displaystyle \begin{aligned}
		\|A_4\|_{E^*}
		&\lesssim \left(\int_{B_{2\de}(p_0)}|\La|^{\frac{2m}{m+1}}|\bar\psi_\vr|^{\frac{2m}{m+1}}d\vol_\ig \right)^{\frac{m+1}{2m}}
		\lesssim \left(\int_{|x|\leq2\de}|x|^{\frac{4m}{m+1}}|\psi_\vr|^{\frac{2m}{m+1}}dx \right)^{\frac{m+1}{2m}}  \\[0.3em]
		&\lesssim \left(\vr^{\frac{6m}{m+1}}\int_0^{\frac{2\de}\vr}\frac{r^{\frac{4m}{m+1}+m-1}}{(1+r^2)^{\frac{m(m-1)}{m+1}}}dr \right)^{\frac{m+1}{2m}}
		\lesssim  \begin{cases}
			\vr^{\frac{m-1}2} &\text{if } 2\leq m\leq6\\
			\vr^3|\ln\vr|^{\frac47} &\text{if } m=7\\
			\vr^3 &\text{if } m\geq8
		\end{cases}
	\end{aligned}$\\

\noindent where we used $\La=-\frac16\sum_k\big(R_{\al k,\bt}x^\al x^\bt +O(r^3)\big)e_k=\sum_kO(r^2)e_k$, \\

	$\displaystyle \begin{aligned}
		\|A_5\|_{E^*}
		&\lesssim \left(\int_{|x|\leq2\de}|x|^{\frac{6m}{m+1}}|\nabla\psi_\vr|^{\frac{2m}{m+1}}dx \right)^{\frac{m+1}{2m}}
		\lesssim \left(\vr^{\frac{6m}{m+1}}\int_0^{\frac{2\de}\vr}\frac{r^{\frac{6m}{m+1}+m-1}}{(1+r^2)^{\frac{m^2}{m+1}}}dr \right)^{\frac{m+1}{2m}} \\[0.3em]
		&\leq \left(\vr^{\frac{6m}{m+1}}\int_0^{\frac{2\de}\vr}\frac{r^{\frac{4m}{m+1}+m-1}}{(1+r^2)^{\frac{m(m-1)}{m+1}}}dr \right)^{\frac{m+1}{2m}}
		\lesssim  \begin{cases}
			\vr^{\frac{m-1}2}  &\text{if } 2\leq m\leq6\\
			\vr^3|\ln\vr|^{\frac47} &\text{if } m=7\\
			\vr^3 &\text{if } m\geq8
		\end{cases}
	\end{aligned}$\\
	%\end{equation}
	
	\noindent Here we used Lemma \ref{lem: zero identity1}, the inequality $|\nabla\psi(x)|\lesssim (1+|x|^2)^{-\frac m2}$ and the same estimate as for $\|A_4\|_{E^*}$. Finally there holds:\\
	
	%\begin{equation}\label{A6}
	$\displaystyle
	\|A_6\|_{E^*}
	\lesssim \left(\int_{|x|\leq2\de}|x|^{\frac{6m}{m+1}}|\psi_\vr|^{\frac{2m}{m+1}}dx \right)^{\frac{m+1}{2m}}
	\lesssim  \begin{cases}
		\vr^{\frac{m-1}2} &\text{if } 2\leq m\leq8\\
		\vr^4|\ln\vr|^{\frac59} &\text{if } m=9\\
		\vr^4 &\text{if } m\geq10
	\end{cases}
	$\\
	%\end{equation}
	
	\noindent Here we used $|\nabla\eta(x)|\lesssim |x|$ and the same estimate as for $\|A_3\|_{E^*}$.
	
	From these estimates we finally obtain:
	\[
	\|\bar R_\vr\|_{E^*}
	\lesssim  \begin{cases}
		\vr^{\frac{m-1}2} &\text{if } 2\leq m\leq 6 \\
		\vr^3|\ln\vr|^{\frac47} &\text{if } m=7 \\
		\vr^3 &\text{if } m\geq8
	\end{cases}
	\]
\end{proof}

\begin{Lem}\label{energy-esti}
	Let $\bar\va_\vr\in\mbs(V)$ be as above, but with a particular choice of $\Psi_0$ in \eqref{t0}, and let $\om_m$ stand for the volume of the standard sphere $S^m$. If the constant $\tau$ in Proposition \ref{conformal normal coordinates} is chosen so that $\tau>-\frac1{90}$, then
	\[
		\aligned
		&\frac12\int_{M}(\bar D\bar\va_\vr,\bar\va_\vr)d\vol_\ig - \frac1{2^*}\int_M|\bar\va_\vr|^{2^*}d \vol_\ig  \\
		&\qquad   \leq \frac1{2m}\left(\frac{m}2\right)^m\om_m+\begin{cases}
			- C|W_\ig(p_0)|^2\vr^4|\ln\vr| & \text{if } m=4 \\
			- C|W_\ig(p_0)|^2\vr^4 & \text{if } m\geq5
		\end{cases} + \begin{cases}
			O(\vr^{m-1})  &\text{if } 4\leq m\leq 5, \\
			O(\vr^5|\ln\vr|) &\text{if } m=6,\\
			O(\vr^5) &\text{if } m\geq7,
		\end{cases}
	\endaligned
	\]
for some constant $C>0$.
\end{Lem}

\begin{proof}
	Analogously to the arguments in Lemma~\ref{derivative-esti}, we shall use \eqref{cut-off spinor identity} and
	\eqref{identity-D} in order to get
	\[
	\int_M(\bar D\bar\va_\vr,\bar\va_\vr)d\vol_\ig = J_1+J_2+\dots+J_7
	\]
	where
	\begin{eqnarray*}
		&&J_1 = \real\int_M \eta\cdot(\ov{\nabla\eta\cdot\psi_\vr},\bar\psi_\vr)d\vol_\ig \\
		&&J_2 = \int_M|\bar\va_\vr|^{2^*}d\vol_\ig \\
		&&J_3 = \int_M(\eta-\eta^{2^*-1})\cdot\eta\cdot|\bar\psi_\vr|^{2^*}d\vol_\ig \\
		&&J_4 = \real\int_M\eta^2\cdot(\Theta\cdot\bar\psi_\vr,\bar\psi_\vr)d\vol_\ig \\
		&&J_5 = \real\int_M\eta^2\cdot(\La\cdot\bar\psi_\vr,\bar\psi_\vr)d\vol_\ig \\
		&&J_6 = \real\sum_{i,j}\int_M\eta^2(b_{ij}-\de_{ij})(\ov{\pa_i\cdot\nabla_{\pa_j}\psi_\vr},\bar\psi_\vr)d\vol_\ig \\
		&&J_7 = \real\sum_{i,j}\int_M\eta\cdot(b_{ij}-\de_{ij})\pa_j\eta\cdot(\ov{\pa_i\cdot\psi_\vr},\bar\psi_\vr)d\vol_\ig.
	\end{eqnarray*}
	
	{In what follows, when no confusion can arise, an index occurs once in an upper (superscript) and once in a lower (subscript) position in a term implies summation of that term over all the values of the index (that is we keep using the Einstein summation convention). And the summations with respect to the indices only in lower (subscript) positions will still use the $\Sigma$-notation. To start with, we notice that $J_1=0$. And, by using the conformal normal coordinates in Proposition \ref{conformal normal coordinates}, we have:}
	
	%\begin{\equ}\label{J1}
	%\\
	%\end{\equ}
	
	\[
 \begin{aligned}%\label{J2}
		J_2 &\leq\int_{|x|\leq\de}|\psi_\vr|^{2^*}dx+{O\left(\int_{\de<|x|\leq2\de}(1+|x|^4)|\psi_\vr|^{2^*}dx \right)} \\[0.3em]
		&\qquad  {+ \frac\tau2\sum_{i,d}R_{i\al \bt d}R_{ikld}\int_{|x|\leq\de}x^\al x^\bt x^k x^l |\psi_\vr|^{2^*}dx+ O\left(\int_{|x|\leq2\de}|x|^5|\psi_\vr|^{2^*}dx \right) }
	\end{aligned}
	\]
	
	\noindent
	{where we have used the Taylor expansion of $\sqrt{\det\ig_{ij}}$ (it can be induced directly from \eqref{conformal-normal-det}). Let us first evaluate the term with curvature tensors. By using $R_{ij}=0$, \eqref{t2} and \eqref{test spinor}, we can easily find (by combination of the symmetry of the integrand and the symmetry of the integral domain) that}
	\[{
	\aligned
	&R_{i\al \bt d}R_{ikld}\int_{|x|\leq\de}x^\al x^\bt x^k x^l |\psi_\vr|^{2^*}dx=\vr^4 m^mR_{i\al \bt d}R_{ikld}\int_{|x|\leq\frac\de\vr}\frac{x^\al x^\bt x^k x^l}{(1+|x|^2)^m}dx \\[0.3em]
	&\quad =\vr^4 m^m\sum_{\al\neq \bt}\big( R_{i\al\bt d}^2+ R_{i\al\bt d}R_{i\bt\al d} \big)\int_{|x|\leq\frac\de\vr}\frac{|x^\al x^\bt|^2}{(1+|x|^2)^m}dx \\[0.3em]
	&\qquad + \vr^4 m^m R_{i\al\al d}^2\int_{|x|<\frac\de\vr} \frac{|x^\al|^4}{(1+|x|^2)^m}dx + \vr^4m^m\sum_{\al\neq k}R_{i\al\al d}R_{i kk d}\int_{|x|<\frac\de\vr}\frac{|x^\al x^k|^2}{(1+|x|^2)^m}dx\\[0.3em]
	&\quad=\vr^4 m^m\sum_{\al\neq \bt}\big( R_{i\al\bt d}^2+ R_{i\al\bt d}R_{i\bt\al d} \big)\int_{|x|\leq\frac\de\vr}\frac{|x^1 x^2|^2}{(1+|x|^2)^m}dx \\[0.3em]
	&\qquad + \vr^4m^m\sum_\al R_{i\al\al d}^2\Big( \int_{|x|<\frac\de\vr} \frac{|x^1|^4-|x^1x^2|^2}{(1+|x|^2)^m}dx \Big),
	\endaligned }
	\]
	{where in the last equality we have used $\sum_{k=1,\, k\neq \al}^mR_{ikkd}=-R_{i\al\al d}$ (i.e. $R_{id}=0$) and the fact that those integrals are independent of the indices of the coordinates. Notice that, by Cauchy-Schwarz inequality, there hold}
	\begin{\equ}\label{curvature-tensor-inequ}
	{\sum_{\al\neq \bt}\big( R_{i\al\bt d}^2+ R_{i\al\bt d}R_{i\bt\al d} \big)=\sum_{\al>\bt}\big(R_{i\al\bt d}^2 + R_{i\bt\al d}^2 +2 R_{i\al\bt d}R_{i\bt\al d}\big)\geq0 }
	\end{\equ}
	and 
	\begin{\equ}\label{integral-inequ}
{	\int_{|x|\leq\frac\de\vr}\frac{|x^1|^4 }{(1+|x|^2)^m} dx=\frac12 \int_{|x|\leq\frac\de\vr}\frac{|x^1|^4+|x^2|^4 }{(1+|x|^2)^m} dx>
	\int_{|x|\leq\frac\de\vr}\frac{|x^1 x^2|^2 }{(1+|x|^2)^m} dx. }
	\end{\equ}
	{Hence the $4$th-order term in the upper bound of $J_2$ depends linearly on $\tau$, and we can derive the following refined estimates}
	
	$\displaystyle \begin{aligned}
		J_2&\leq m^m\om_{m-1}\int_0^{\frac\de\vr}\frac{r^{m-1}}{(1+r^2)^m}dr { \,+ \frac{\vr^4m^m\tau}{2}\sum_{i,d}R_{i\al\bt d}R_{ikld}\int_{|x|\leq\frac\de\vr}\frac{x^\al x^\bt x^k x^l}{(1+|x|^2)^m}dx } \\[0.3em]
	&\qquad 	+ O\left(\int_{\frac\de\vr}^{\frac{2\de}\vr}\frac{r^{m-1}}{(1+r^2)^m}dr \right) 
		+ {O\left(\vr^5\int_0^{\frac{2\de}\vr}\frac{r^{m+4}}{(1+r^2)^m}dr \right) }\\[0.3em]
		&= m^m\om_{m-1}\int_0^{\infty}\frac{r^{m-1}}{(1+r^2)^m}dr  { \,+ \frac{\vr^4m^m\tau}{2}\sum_{i,d}R_{i\al\bt d}R_{ikld}\int_{|x|\leq\frac\de\vr}\frac{x^\al x^\bt x^k x^l}{(1+|x|^2)^m}dx } \\[0.3em]
		&\qquad + O(\vr^m)
		+ {\begin{cases}
			O(\vr^m), & \text{if } 2\leq m\leq 4 \\
			O(\vr^5|\ln\vr|), &\text{if } m=5 \\
			O(\vr^5), &\text{if } m\geq6
		\end{cases} }\\
		&= m^m\om_{m-1}\int_0^{\infty}\frac{r^{m-1}}{(1+r^2)^m}dr { \,+ \frac{\vr^4m^m\tau}{2}\sum_{i,d}R_{i\al\bt d}R_{ikld}\int_{|x|\leq\frac\de\vr}\frac{x^\al x^\bt x^k x^l}{(1+|x|^2)^m}dx } \\[0.3em]
		&\qquad + {\begin{cases}
			O(\vr^m), & \text{if } 2\leq m\leq 4 \\
			O(\vr^5|\ln\vr|), &\text{if } m=5 \\
			O(\vr^5), &\text{if } m\geq 6
		\end{cases}  } %\quad \text{here } N\geq5
	\end{aligned}$\\
	
	\noindent
	and $J_3\lesssim \int_{\de\leq|x|\leq2\de}|\psi_\vr|^{2^*}dx \lesssim \vr^m$.\\

To estimate $J_4$, we first note that 
	\[
	\aligned
	\Theta&=-\frac1{144}\sum_{\substack{i,j,k \\ i\neq j\neq k\neq i}}\sum_l R_{l\bt\ga k}(R_{ji\al l}+R_{jl\al i})x^\al x^\bt x^\ga e_i\cdot e_j \cdot e_k +O(r^4) \\
	&=\sum_{\substack{i,j,k \\ i\neq j\neq k\neq i}}A_{ijk\al\bt\ga}x^\al x^\bt x^\ga e_i\cdot e_j\cdot e_k + O(r^4)
	\endaligned
	\]
	and 
	\[
	J_4=\real\int_{B_\de(p_0)} + \real\int_{B_{2\de}(p_0)\setminus B_\de(p_0)}\eta^2\cdot(\Theta\cdot\bar\psi_\vr,\bar\psi_\vr)d\vol_\ig.
	\]	
	We have
	\[
	\aligned
	\real\int_{B_\de(p_0)}(\Theta\cdot\bar\psi_\vr,\bar\psi_\vr)d\vol_\ig&=\sum_{\substack{i,j,k \\ i\neq j\neq k\neq i}}A_{ijk\al\bt\ga}\real\int_{|x|\leq\de}x^\al x^\bt x^\ga (e_i\cdot e_j\cdot e_k\cdot\psi_\vr,\psi_\vr)dx  \\
	&\qquad + { O\left( \int_{|x|\leq\de} |x|^{4}|\psi_\vr|^2 dx\right)  }.
	\endaligned
	\]
	Using the explicit formula of $\psi$ in \eqref{t0}, we get
	\[
	\aligned
	&A_{ijk\al\bt\ga}\real\int_{|x|\leq\de}x^\al x^\bt x^\ga (e_i\cdot e_j\cdot e_k\cdot\psi_\vr,\psi_\vr)dx \\
	&\qquad = -\vr^4A_{ijk\al\bt\ga}\int_{|x|\leq \frac\de\vr}x^\al x^\bt x^\ga \frac{m^{m-1}}{(1+|x|^2)^m}2\real(\pa_i\cdot\pa_j\cdot\pa_k\cdot x\cdot\Psi_0,\Psi_0) dx \\
	&\qquad = \vr^4 \sum_l A_{ijkl} \real(\pa_i\cdot\pa_j\cdot\pa_k\cdot \pa_l\cdot\Psi_0,\Psi_0)
	\endaligned
	\]
	for some coefficients $A_{ijkl}$. And hence, by Lemma \ref{lem: zero identity2}, for a suitable choice of $\Psi_0$, we have
	\[
	J_4={O\left( \int_{\de\leq|x|\leq2\de}|x|^3|\psi_\vr|^2dx \right)} + O\left( \int_{|x|\leq\de} |x|^{4}|\psi_\vr|^2 dx \right).
	\]
	Hence
	\[
	\aligned
	J_4 &\lesssim {\int_{\de\leq|x|\leq2\de}|x|^3|\psi_\vr|^2dx} + \int_{|x|\leq\de}|x|^4|\psi_\vr|^2dx \\[0.3em]
	&\lesssim {\vr^4\int_{\frac\de\vr}^{\frac{2\de}{\vr}}\frac{r^{m+2}}{(1+r^2)^{m-1}}dr}+\vr^5\int_0^{\frac{\de}\vr}\frac{r^{m+3}}{(1+r^2)^{m-1}}dr
	\lesssim \begin{cases}
		\vr^{m-1} &\text{if } 2\leq m\leq 5,  \\
		\vr^5|\ln\vr| &\text{if } m=6,\\
		\vr^5 &\text{if } m\geq7.
	\end{cases}
	\endaligned
	\]

	We also have that $J_5=J_7=0$ since the integrands take pure imaginary values. Then, it only remains to evaluate $J_6$:\\
	
	$\displaystyle \begin{aligned}%\label{J6}
		J_6 &= \real\sum_{i,j}\int_{|x|\leq\de}(b_{ij}-\de_{ij})(\pa_i\cdot\nabla_{\pa_j}\psi_\vr,\psi_\vr)dx+ {O\left( \vr^2\int_{\frac\de\vr\leq|x|\leq\frac{2\de}\vr}\frac{|x|^2(1+|x|)}{(1+|x|^2)^m}dx \right) }\\
		&\qquad + {O\left( \vr^{6}\int_{|x|\leq\frac{\de}\vr}\frac{|x|^{6}(1+|x|)}{(1+|x|^2)^m}dx \right) }\\
		& =\real\sum_{i,j}\int_{|x|\leq\de}(b_{ij}-\de_{ij})(\pa_i\cdot\nabla_{\pa_j}\psi_\vr,\psi_\vr)dx+ {\begin{cases}
			O(\vr^{m-1}) &\text{if } 2\leq m\leq 6,\\
			O(\vr^{6}|\ln\vr|) &\text{if } m=7,\\
			O(\vr^{6}) &\text{if } m\geq 8.
		\end{cases} }
	\end{aligned}$\\
	
	\noindent Note that
	\[
	\aligned
	\real(\pa_i\cdot\nabla_{\pa_j}\psi(x),\psi(x))&=-\frac{m^{m-1}}{(1+|x|^2)^m}\real\big(\pa_i\cdot\pa_j\cdot\Psi_0,\,(1-x)\cdot\Psi_0\big) \\
%	&=\begin{cases}
%		\frac{m^{m-1}}{(1+|x|^2)^m} & \text{if } i=j \\
%		\frac{m^{m-1}}{(1+|x|^2)^m}\real(\pa_i\cdot\pa_j\cdot\Phi_0,x\cdot\Phi_0) & \text{if } i\neq j
%	\end{cases}
\endaligned
	\]
	we compute
	\begin{eqnarray*}
	&&\real\sum_{i,j}\int_{|x|\leq\de}(b_{ij}-\de_{ij})(\pa_i\cdot\nabla_{\pa_j}\psi_\vr,\psi_\vr)dx \\
	&&\qquad =-\frac{\vr^2}6\sum_{i,j}\int_{|x|\leq\frac\de\vr}R_{i\al\bt j}x^\al x^\bt \real(\pa_i\cdot\nabla_{\pa_j}\psi,\psi)dx \\
	&&\qquad \quad -\frac{\vr^3}{12}\sum_{i,j}\int_{|x|\leq\frac\de\vr}R_{i\al\bt j,k}x^\al x^\bt x^k \real(\pa_i\cdot\nabla_{\pa_j}\psi,\psi)dx \\
	&&\qquad \quad
	-\vr^4\sum_{i,j}\Big( \frac1{40}R_{i\al\bt j,kl}-\frac7{360}\sum_dR_{i\al\bt d}R_{jkld} \Big)\int_{|x|\leq\frac\de\vr}x^\al x^\bt x^k x^l \real(\pa_i\cdot\nabla_{\pa_j}\psi,\psi)dx.
	\end{eqnarray*}
	Since
	\[
	R_{i\al\bt j}x^\al x^\bt=R_{i\bt\al j}x^\al x^\bt
	=R_{j\al \bt i}x^\al x^\bt
	\]
	and $\pa_i\cdot\pa_j=-\pa_j\cdot\pa_i$ for $i\neq j$, we have
	\[
	\aligned
	&\sum_{i,j}\int_{|x|\leq\frac\de\vr}R_{i\al\bt j}x^\al x^\bt \real(\pa_i\cdot\nabla_{\pa_j}\psi,\psi)dx\\
	&\qquad =-\sum_{i,j}\int_{|x|\leq\frac\de\vr}\frac{m^{m-1}R_{i\al\bt j}x^\al x^\bt}{(1+|x|^2)^m} \real(\pa_i\cdot\pa_j\cdot\Psi_0,\Psi_0)dx\\
	&\qquad =-\int_{|x|\leq\frac\de\vr}\frac{m^{m-1}R_{\al\bt }x^\al x^\bt}{(1+|x|^2)^m} dx\\
	&\qquad =0
	\endaligned
	\]
	due to $R_{\al\bt}=0$ as was indicated by Proposition \ref{conformal normal coordinates}(1).
	
	Analogously
	\[
	\aligned
	&\sum_{i,j}\int_{|x|\leq\frac\de\vr}R_{i\al\bt j,k}x^\al x^\bt x^k \real(\pa_i\cdot\nabla_{\pa_j}\psi,\psi)dx \\
	&\qquad =\sum_{i,j}\int_{|x|\leq\frac\de\vr}\frac{m^{m-1}R_{i\al\bt j,k}x^\al x^\bt x^k}{(1+|x|^2)^m} \real(\pa_i\cdot\pa_j\cdot\Psi_0,x\cdot\Psi_0)dx \\
	&\qquad =\int_{|x|\leq\frac\de\vr}\frac{m^{m-1}R_{\al\bt ,k}x^\al x^\bt x^k}{(1+|x|^2)^m} \real(\Psi_0,x\cdot\Psi_0)dx =0 \\
	\endaligned
	\]
	since $\real(\Psi_0,x\cdot\Psi_0)=0$ for all $x\in\R^m$.

	Combining these estimates we deduce that
	\[
	\aligned
	J_6 &=-\vr^4\sum_{i}\Big( \frac1{40}R_{i\al\bt i,kl}-\frac7{360}\sum_dR_{i\al\bt d}R_{ikld} \Big)\int_{|x|\leq\frac\de\vr}\frac{m^{m-1}x^\al x^\bt x^k x^l}{(1+|x|^2)^m} dx \\
 &\qquad	+  {\begin{cases}
		O(\vr^{m-1}) &\text{if } 2\leq m\leq 6,\\
		O(\vr^{6}|\ln\vr|) &\text{if } m=7,\\
		O(\vr^{6}) &\text{if } m\geq 8.
	\end{cases} }
\endaligned
	\]
Recall the conformal normal coordinates in Proposition \ref{conformal normal coordinates}, we have that $W_{ijkl}=R_{ijkl}$ and
\begin{\equ}\label{XX1}
{\Big(R_{\al\bt,kl}+\Big(\frac{2}9+20\tau\Big)\sum_{i,d} R_{i\al\bt d}R_{ikld}\Big)x^\al x^\bt x^k x^l=0}
\end{\equ}
at $p_0$. {Notice that $R_{\al\bt,kl}=(\ig^{ij}R_{i\al j\bt})_{,kl}$, we find
\[
R_{\al\bt,kl}=\sum_i R_{i\al i\bt,kl}+\ig^{ij}_{,kl}R_{i\al j\bt}=-\sum_i R_{i\al \bt i,kl}+\frac13\sum_{i,j}(R_{iklj}+R_{ilkj})R_{i\al \bt j},
\]
where we have used 
\[
\ig^{ij}=\de_{ij}-\frac13 R_{i\al\bt j}x^\al x^\bt + O(r^3)
\]
and $R_{i\al j\bt}=-R_{i\al\bt j}$ in the above equality. And hence, \eqref{XX1} can be rewritten as
	\[
	{\sum_i\Big(R_{i\al\bt i,kl}-\Big( \frac89+20\tau \Big)\sum_{d} R_{i\al\bt d}R_{ikld}\Big)x^\al x^\bt x^k x^l=0}
	\]
	 which suggests
	\[
	\aligned
	J_6 &=-\Big(\frac1{360}+\frac\tau2\Big)\vr^4\sum_{i,d}R_{i\al\bt d}R_{ikld}\int_{|x|\leq\frac\de\vr}\frac{m^{m-1}x^\al x^\bt x^k x^l}{(1+|x|^2)^m} dx	+ \begin{cases}
	O(\vr^{m-1}) &\text{if } 2\leq m\leq 6,\\
	O(\vr^{6}|\ln\vr|) &\text{if } m=7,\\
	O(\vr^{6}) &\text{if } m\geq 8.
\end{cases} \\
	\endaligned
	\]	
Similar to the estimate of $J_2$, by using $R_{ij}=0$, we have
\begin{\equ}\label{E1}
\aligned
&R_{i\al\bt d}R_{ikld}\int_{|x|\leq\frac\de\vr}\frac{m^{m-1}x^\al x^\bt x^k x^l}{(1+|x|^2)^m} dx \\
&\quad = \sum_{\al\neq\bt}\big(R_{i\al\bt d}^2+R_{i\al\bt d}R_{i\bt\al d}\big)\int_{|x|\leq\frac\de\vr}\frac{m^{m-1}|x^\al x^\bt|^2 }{(1+|x|^2)^m} dx \\
&\qquad + \sum_\al R_{i\al\al d}^2\Big( \int_{|x|\leq\frac\de\vr}\frac{m^{m-1}\big(|x^1|^4 - |x^1x^2|^2\big) }{(1+|x|^2)^m}  dx \Big).
\endaligned
\end{\equ}
By virtue of \eqref{curvature-tensor-inequ} and \eqref{integral-inequ},  we can find a constant $C_1>0$ depends only on the dimension $m$ such that 
\[
\aligned
&R_{i\al\bt d}R_{ikld}\int_{|x|\leq\frac\de\vr}\frac{m^{m-1}x^\al x^\bt x^k x^l}{(1+|x|^2)^m} dx \\
&\quad \geq  C_1 min \Big\{\int_{|x|\leq\frac\de\vr}\frac{|x^1x^2|^2 }{(1+|x|^2)^m} dx,\, 
\int_{|x|\leq\frac\de\vr}\frac{\big(|x^1|^4 - |x^1x^2|^2\big) }{(1+|x|^2)^m}  dx
 \Big\} \sum_{\al,\bt} R_{i\al\bt d}R_{i\bt\al d},
 \endaligned
\]
where the integrals can be estimated via elementary calculations. Noticing
\[
\sum_{\al,\bt,d}R_{i\al\bt d}R_{i\bt\al d}=\sum_{\al,d,\bt}R_{i\al d\bt}R_{id\al \bt}=\frac12\sum_{\al,\bt,d}R_{i\al\bt d}\big(R_{i\bt\al d}-R_{id\al \bt} \big)=\frac12\sum_{\al,\bt,d}R_{i\al\bt d}^2
\]
by the symmetry of the curvature tensor and the Bianchi identity}, we then have
\begin{eqnarray*}
\Big(\frac12-\frac1{2^*}\Big)J_2+\frac12J_6 &\leq&  \frac{m^m\om_{m-1}}{2m}\int_0^{\infty}\frac{r^{m-1}}{(1+r^2)^m}dr  \\
& &{ \,- \frac1{720}\vr^4\sum_{i,d}R_{i\al\bt d}R_{ikld}\int_{|x|\leq\frac\de\vr}\frac{m^{m-1}x^\al x^\bt x^k x^l}{(1+|x|^2)^m}dx } \\[0.3em]
&&\qquad + {\begin{cases}
		O(\vr^{m-1}), & \text{if } 4\leq m\leq 5 \\
		O(\vr^5), &\text{if } m\geq 6
\end{cases}  } \\
&\leq&\frac{m^m\om_{m-1}}{2m}\int_0^{\infty}\frac{r^{m-1}}{(1+r^2)^m}dr  + \begin{cases}
	- C|W_\ig(p_0)|^2\vr^4|\ln\vr| & \text{if } m=4 \\
	- C|W_\ig(p_0)|^2\vr^4 & \text{if } m\geq5
\end{cases}	\\
& &+ { \begin{cases}
	O(\vr^{m-1}) &\text{if } 4\leq m\leq 5\\
	O(\vr^{5}) &\text{if } m\geq 6
\end{cases} }
\end{eqnarray*}
where $C>0$ is a constant depending only on the dimension. 

Finally, the assertion follows upon taking into account that
\[
\om_m=\int_{\R^m}\Big( \frac2{1+|x|^2} \Big)^mdx=2^m\om_{m-1}\int_0^\infty \frac{r^{m-1}}{(1+r^2)^m}dr
\]
in the computation of $J_2$.
\end{proof}

%\begin{Rem}
%\rx{Let us remind here that in the estimation of $J_2$, the dropped fourth-order term is actually proportional to \eqref{E1} with only a constant factor and also has a negative sign, indicating that they can be combined. The emphasis here is that, as was pointed out in Remark \ref{rmk-conformal-normal-coordinates}, the leading term in the expansion of $J_6$ will not be negative if one uses a metric whose determinant vanishes to a sufficiently large order.}
%\end{Rem}

\subsection{For the case $\ker(D_\ig)=\{0\}$}

Now let us apply Theorem \ref{abstract thm} to the functional 
\[
\Phi(\psi)=\frac12\big( \|\psi^+\|^2-\|\psi^-\|^2 \big)-\frac1{2^*}\|\psi\|_{L^{2^*}}^{2^*} \quad \text{for } \psi\in E=H^{\frac12}(M,\mbs(M)).
\]
In fact, since $\ker(D_\ig)=\{0\}$, we have the orthogonal decomposition $E=E^+\op E^-$ with $E^\pm$ being the positive and negative eigenspaces of $D_\ig$ respectively, and the inner product is given by
\[
\inp{\psi}{\va}=\real(|D_\ig|^{\frac12}\psi,|D_\ig|^{\frac12}\va)_2 \quad \text{for } \psi,\va\in E.
\]
For the super-quadratic term 
\[
\Psi(\psi)=\frac1{2^*}\|\psi\|_{L^{2^*}}^{2^*},
\] 
we can compute
\[
\inp{\nabla\Psi(\psi)}{\va}=\real\int_M|\psi|_\ig^{2^*-2}(\psi,\va)_\ig d\vol_{\ig}
\]
and
\[
\inp{\nabla^2\Psi(\psi)[\va]}{\phi}=\real\int_M|\psi|_\ig^{2^*-2}(\phi,\va)_\ig d\vol_{\ig}+(2^*-2)\int_M|\psi|_\ig^{2^*-4}\real(\psi,\phi)_\ig\real(\psi,\va)_\ig d\vol_{\ig}.
\]
Hence, it follows directly that $\Psi$ satisfies the hypotheses (H1)-(H4). And
\[
\aligned
&\inp{\nabla^2\Psi(\psi)[\psi+\va]}{\psi+\va}-2\inp{\nabla\Psi(\psi)}{\va} \\
&\quad =\int_M|\psi|_\ig^{2^*-2}\Big[ |\va|_\ig^2+(2^*-1)|\psi|_\ig^2+2(2^*-2)\real(\psi,\va)_\ig+(2^*-2)\frac{|\real(\psi,\va)_\ig|^2}{|\psi|_\ig^2} \Big]d\vol_{\ig} \\
&\quad \geq\int_M|\psi|_\ig^{2^*-2}\Big[ (2^*-1)|\psi|_\ig^2+2(2^*-2)\real(\psi,\va)_\ig+(2^*-1)\frac{|\real(\psi,\va)_\ig|^2}{|\psi|_\ig^2} \Big]d\vol_{\ig} \\
&\quad \geq\Big((2^*-1)-\frac{(2^*-2)^2}{2^*-1} \Big)\int_M|\psi|_\ig^{2^*}d\vol_{\ig}=\frac{m+3}{m+1}\inp{\nabla\Psi(\psi)}{\psi}
\endaligned
\] 
where we used the fundamental inequalities
\[
|\va|_\ig^2\geq \frac{|\real(\psi,\va)_\ig|^2}{|\psi|_\ig^2} \quad \text{and} \quad
2|\real(\psi,\va)_\ig|\leq \frac{2^*-2}{2^*-1}|\psi|_{\ig}^2+\frac{2^*-1}{2^*-2}\frac{|\real(\psi,\va)_\ig|^2}{|\psi|_\ig^2}
\]
for $|\psi|_\ig\neq0$. This implies that $\Psi$ satisfies (H5).

Therefore, by Lemmas \ref{derivative-esti}, \ref{energy-esti} and a direct application of Theorem \ref{abstract thm}, we obtain the energy estimate 
\[
\aligned
c&=\inf_{\va\in E^+\setminus\{0\}}\max_{\psi\in\R^+\va\op E^-}\Phi(\psi)\\
&\leq\frac1{2m}\left(\frac{m}2\right)^m\om_m- C|W_\ig(p_0)|^2\vr^4 + \begin{cases}
	O(\vr^5|\ln\vr|) &\text{if } m=6,\\
	O(\vr^5) &\text{if } m\geq7.
\end{cases}
\endaligned
\]
From this, together with the compactness result Lemma \ref{PS Dirac}, we can see that $c<\ga_{crit}$ is a critical value of $\Phi$ and is attained.

\subsection{For the case $\ker(D_\ig)\neq\{0\}$}

In this case, we have
\[
E=H^{\frac12}(M,\mbs(M))=E^+\op E^0\op E^-
\]
where $E^\pm$ are positive and negative eigenspaces of $D_\ig$, respectively, and $E^0=\ker(D_\ig)$.
To deal with the non-trivial kernel, let us take
\[
T:L^{2^*}=L^{2^*}(M,\mbs(M))\to E^0, \quad T(\psi) = \arg\min_{\phi\in E^0} \|\psi-\phi\|_{L^{2^*}}^{2^*},
\]
i.e.\ $T(\psi)\in E^0$ is the best approximation of $\psi\in L^{2^*}$ in $E^0$. This exists and is unique because the $L^{2^*}$ norm is uniformly convex. Of course, $T$ is not linear in general. Evidently, we have $\|\psi\|_{L^{2^*}}^{2^*}\geq\|\psi-T(\psi)\|_{L^{2^*}}^{2^*}$ and therefore
\begin{equation}\label{def:tilde-e}
	\Phi(\psi) \le \widetilde\Phi(\psi) := \Phi(\psi-T(\psi))
	= \frac12\big(\|\psi^+\|^2-\|\psi^-\|^2\big) -\frac1{2^*}\int_M|\psi-T(\psi)|_\ig^{2^*}d\vol_\ig
\end{equation}
for all $\psi\in E$. Set
\[
\widetilde\Psi(\psi)=\frac1{2^*}\int_M|\psi-T(\psi)|_\ig^{2^*}d\vol_\ig
\]
we have the follow lemma.

\begin{Lem}\label{lemma widetilde-psi}
	\begin{itemize}
		\item[$(1)$] For $\psi\in E$, $\phi\in E^0$ and $t\in\R$, we have $T(t\psi)=tT(\psi)$, $T(\psi+\phi)=T(\psi)+\phi$ and $\widetilde\Psi(\psi+\phi)=\widetilde\Psi(\psi)$.
		
		\item[$(2)$] $T$ is $C^1$ on $E^+\op E^-\setminus\{0\}$ and $\nabla T(\psi)[\psi]=T(\psi)$ for $\psi\in E^+\op E^-\setminus\{0\}$.
		
		\item[$(3)$] $\widetilde\Psi$ is $C^2$ on $E^+\op E^-$ and is convex, moreover, $P^0\nabla\widetilde\Psi(\psi)=0$ for all
		$\psi\in E^+\op E^-$, where $P^0: E\to E^0$ is the orthogonal projector.

		\item[$(4)$] For any $\psi,\va\in E^+\op E^-$,
		\[
		\inp{\nabla^2\widetilde\Psi(\psi)[\psi+\va]}{\psi+\va}-2\inp{\nabla\widetilde\Psi(\psi)}{\va}\geq \frac{m+3}{m+1}\inp{\nabla\widetilde\Psi(\psi)}{\psi}.
		\]
		
		\item[$(5)$] Critical points of $\widetilde\Phi$ on $E^+\op E^-$ are in one-to-one correspondence with critical points of $\Phi$ on $E$ via the injective map $\psi\mapsto \psi-T(\psi)$ from $E^+\op E^-$ into $E$.
	\end{itemize}
\end{Lem}
\begin{proof}
Since $(1)$ is a straightforward consequence of the definition of $T$, we start with the proof of $(2)$.

Let us fix a $\psi\in E^+\op E^-\setminus\{0\}$ and consider the functional $\Psi_\psi$ defined on $E^0$ by
\[
\Psi_\psi(\phi)=\frac1{2^*}\int_M|\psi-\phi|_\ig^{2^*}d\vol_\ig, \quad \phi\in E^0.
\]
Then we have $\Psi_\psi$ is $C^2$  and $\Psi_\psi(\phi)\geq\frac1{2^*}( |\phi|_{2^*}-|\psi|_{2^*} )^{2^*}$, that is $\Psi_\psi$ is coercive. Moreover, for $\phi,\va\in E^0$, we have
\[
\inp{\nabla\Psi_\psi(\phi)}{\va}=-\real\int_M|\psi-\phi|_\ig^{2^*-2}(\psi-\phi,\va)_\ig d\vol_\ig
\]
and
\[
\aligned
\inp{\nabla^2\Psi_\psi(\phi)[\va]}{\va}&=\int_M\Big[|\psi-\phi|_\ig^{2^*-2}|\va|_\ig^2+(2^*-2)|\psi-\phi|_\ig^{2^*-4}|\real(\psi-\phi,\va)_\ig|^2\, \Big]d\vol_\ig\\[0.3em]
&\geq\int_M|\psi-\phi|_\ig^{2^*-2}|\va|_\ig^2d\vol_\ig.
\endaligned
\]
Hence $\Psi_\psi$ is convex.
And since $E^0$ is a finite dimensional eigenspace of the Dirac operator, by the convexity and coerciveness of $\Psi_\psi$, we have $T(\psi)$ is the unique strict minimum point for $\Psi_{\psi}$ on $E^0$, which is also the only critical point of $\Psi_{\psi}$. By virtue of the above observation and B\"ar's theorem \cite{Bar1997} on nodal sets of the eigenspinors, we can see that the self-adjoint operator 
\[
\va\mapsto\nabla^2\Psi_\psi(T(\psi))[\va], \quad \forall \va\in E^0
\]
is invertible. Hence, the implicit function theorem implies $T:E^+\op E^-\setminus\{0\}\to E^0$ is of $C^1$.

Although $(1)$ implies that $T$ cannot be differentiable at $0$, we shall show that $\widetilde\Psi$ is still of class $C^2$. In fact, by $\widetilde\Psi(0)=0$, we find that
\[
\widetilde\Psi(t\psi)-\widetilde\Psi(0)=\frac1{2^*}\int_M|t\psi-T(t\psi)|_\ig^{2^*}d\vol_\ig=\frac{t^{2^*}}{2^*}\int_M|\psi-T(\psi)|_\ig^{2^*}d\vol_\ig
\]
for $\psi\neq0$. Hence, combining the fact $2^*>2$, we see  that $\nabla\widetilde\Psi(0)=0$. 
For $\psi\in E^+\op E^-\setminus\{0\}$, it follows directly that $P^0\nabla\widetilde\Psi(\psi)=0$ and that
\[
\aligned
\inp{\nabla\widetilde\Psi(\psi)}{\va}&=\real\int_M|\psi-T(\psi)|^{2^*-2}(\psi-T(\psi),\va-\nabla T(\psi)[\va])_\ig d\vol_\ig \\[0.3em]
&=\real\int_M|\psi-T(\psi)|^{2^*-2}(\psi-T(\psi),\va-T(\va))_\ig d\vol_\ig \\[0.3em]
&=\real\int_M|\psi-T(\psi)|^{2^*-2}(\psi-T(\psi),\va)_\ig d\vol_\ig
\endaligned
\]
for all $\va\in E^+\op E^-$. Then we have the derivative of $\widetilde\Psi$ is continuous in $\psi$, and hence $\widetilde\Psi$ is of class $C^1$ on $E^+\op E^-$.

For the $C^2$-smoothness, we only need to show that the second derivative of $\widetilde\Psi$ exists and is continuous at $0$. To this end, let us take again $\psi\in E^+\op E^-\setminus\{0\}$ and $\va\in E^+\op E^-$. For the second derivative at $0$,  we first note that
\[
\inp{\nabla\widetilde\Psi(t\psi)}{\va}=t^{2^*-1}\real\int_M|\psi-T(\psi)|_\ig^{2^*-2}(\psi-T(\psi),\va)_\ig d\vol_\ig.
\]
Since $2^*>2$ and $\nabla\widetilde\Psi(0)=0$, this implies that $\widetilde\Psi$ is twice differentiable at $0$ and $\nabla^2\widetilde\Psi(0)=0$. For $\psi\neq0$, by taking the derivative in the direction of $\va$ at both sides of $P^0\nabla\widetilde\Psi(\psi)=0$, we have $P^0\nabla^2\widetilde\Psi(\psi)[\va]=0$. Hence, we see that
\[
\aligned
\inp{\nabla^2\widetilde\Psi(\psi)[\phi]}{\va}&=\frac2{m-1}\int_M|\psi-T(\psi)|_\ig^{2^*-4}\real(\psi-T(\psi),\phi)_\ig\real(\psi-T(\psi),\va)_\ig d\vol_\ig  \\
&\quad + \real\int_M|\psi-T(\psi)|_\ig^{2^*-2}(\phi,\va)_\ig d\vol_\ig
\endaligned
\]
for $\phi,\va\in E^+\op E^-$. It follows that $\nabla^2\widetilde\Psi$ is continuous when $\psi$ approaches $0$, and hence $\widetilde\Psi$ is of class $C^2$.

The convexity of $\widetilde\Psi$ is an immediate consequence of the above calculation and, furthermore, we have
\[
\aligned
&\inp{\nabla^2\widetilde\Psi(\psi)[\psi+\va]}{\psi+\va}-2\inp{\nabla\widetilde\Psi(\psi)}{\va} \\
&\quad =\int_M\Big[|\psi-T(\psi)|_\ig^{2^*-2}\big|\psi+\va-\nabla T(\psi)[\psi+\va]\big|_\ig^2 \\
&\qquad +(2^*-2)|\psi-T(\psi)|_\ig^{2^*-4}\big|\real\big(\psi-T(\psi),\psi+\va-\nabla T(\psi)[\psi+\va]\big)\big|_\ig^2 \Big]d\vol_\ig \\
&\qquad -2\real\int_M|\psi-T(\psi)|_\ig^{2^*-2}(\psi-T(\psi),\va-\nabla T(\psi)[\va])_\ig d\vol_\ig.
\endaligned
\]
For simplicity, we denote $z=\psi-T(\psi)$ and $w=\va-\nabla T(\psi)[\va]$. Then, using the fact $\nabla T(\psi)[\psi]=T(\psi)$, we can see from the above equality that
\[
\aligned
&\inp{\nabla^2\widetilde\Psi(\psi)[\psi+\va]}{\psi+\va}-2\inp{\nabla\widetilde\Psi(\psi)}{\va} \\
&\quad = \int_M\Big[ |z|_\ig^{2^*-2}|z+w|_\ig^2+(2^*-2)|z|_\ig^{2^*-4}|\real(z,z+w)_\ig|^2-2 |z|_\ig^{2^*-2}\real(z,w)_\ig \Big]d\vol_\ig \\
&\quad=\int_M|z|_\ig^{2^*-2}\Big[|w|_\ig^2+(2^*-1)|z|_\ig^2 +2(2^*-2)\real(z,w)_\ig+(2^*-2)\frac{|\real(z,w)_\ig|^2}{|z|_\ig^2}\Big] d\vol_\ig \\
&\quad \geq\int_M|z|_\ig^{2^*-2}\Big[(2^*-1)|z|_\ig^2 +2(2^*-2)\real(z,w)_\ig+(2^*-1)\frac{|\real(z,w)_\ig|^2}{|z|_\ig^2}\Big] d\vol_\ig \\
&\quad \geq \Big( (2^*-1)-\frac{(2^*-2)^2}{2^*-1} \Big)\int_M|z|_\ig^{2^*}d\vol_\ig = \frac{m+3}{m+1}\int_M|\psi-T(\psi)|_\ig^{2^*}d\vol_\ig
\endaligned
\]
where we used again the fundamental inequalities
\[
|w|_\ig^2\geq \frac{|\real(z,w)_\ig|^2}{|z|_\ig^2} \quad \text{and} \quad
2|\real(z,w)_\ig|\leq \frac{2^*-2}{2^*-1}|z|_\ig^2+ \frac{2^*-1}{2^*-2}\frac{|\real(z,w)_\ig|^2}{|z|_\ig^2}
\]
for $|z|_\ig\neq0$. This proves $(4)$. 

Finally, to see $(5)$, we first mention that if $\va\in H^{1/2}(M,\mbs(M))$ is a critical point of $\Phi$ then $\real\int_M|\va|_\ig^{2^*-2}(\va,\phi)_\ig d\vol_\ig=0$ for all $\phi\in E^0$. This implies $P^0\va=-T(\va)$ and $\va-P^0\va\in E^+\op E^-$ is a critical point of $\widetilde\Phi$. On the other hand, if $\psi\in E^+\op E^-$ is a critical point of $\widetilde\Phi$, it follows directly from the expression of $\nabla\widetilde\Psi$ that $\psi-T(\psi)$ %is a critical point of $\widetilde\Phi$ on $H^{1/2}(M,\mbs(M))$ and hence 
is a critical point of $\Phi$. This completes the whole proof.
\end{proof}

Thanks to Lemma \ref{lemma widetilde-psi}, we can reduce the problem to the search of critical points of $\widetilde\Phi$ on the closed subspace $E^+\op E^-\subset H^{1/2}(M,\mbs(M))$. And it can be easily seen from the computations in the proof of Lemma \ref{lemma widetilde-psi} that $\widetilde\Psi$ satisfies the hypotheses (H1)-(H5) as we can use the explicit expressions of $\nabla\widetilde\Psi$ and $\nabla^2\widetilde\Psi$. Now we address our main inequality using our test spinor $\bar\va_\vr$ in \eqref{test spinor}. To start with, we have
\begin{\equ}\label{l1-norm}
	\|\bar\va_\vr\|_{L^1}=\int_M|\bar\va_\vr|_\ig d\vol_\ig \lesssim \vr^{\frac{m-1}2}
	\qquad\text{and}\qquad
	\|\bar\va_\vr\|_{L^{2^*-1}}^{2^*-1}=\int_M|\bar\va_\vr|_\ig^{2^*-1}d\vol_\ig \lesssim \vr^{\frac{m-1}2}.
\end{\equ}
As an immediate consequence of these two estimates, jointly with the $L^{2^*}$-boundedness of $\bar\va_\vr$ (see in the proof of Lemma \ref{energy-esti} the estimate of $J_2$), we have that $\bar\va_\vr\rightharpoonup 0$ in $L^{2^*}(M,\mbs(M))$. By Lemma \ref{lemma widetilde-psi} $(3)$, we can get that
\begin{\equ}\label{T-va-0}
	 \real \int_M|\bar\va_\vr-T(\bar\va_\vr)|_\ig^{2^*-2}(\bar\va_\vr-T(\bar\va_\vr), T(\bar\va_\vr))_\ig d\vol_\ig \equiv 0 
\end{\equ}
for all $\vr>0$. And hence, we can deduce from the boundedness of $T(\va_\vr)$ in $E^0$ and $\dim E^0<\infty$ that $T(\bar\va_\vr)\to0$ as $\vr\to0$ (since $E^0$ is of finite dimension, this convergence will be valid in any norm). Furthermore, we can also deduce from the definition of $\bar\va_\vr$ and \eqref{T-va-0}  that
\[
\aligned
\int_{M\setminus B_{2\de}(p_0)}|T(\bar\va_\vr)|_\ig^{2^*}d\vol_\ig &\leq \int_M|\bar\va_\vr-T(\bar\va_\vr)|_\ig^{2^*-2}|T(\bar\va_\vr)|_\ig^2d\vol_\ig \\
&=\real\int_M|\bar\va_\vr-T(\bar\va_\vr)|_\ig^{2^*-2}(\bar\va_\vr, T(\bar\va_\vr))_\ig d\vol_\ig.
\endaligned
\]
This implies that
\[
\aligned
\|T(\bar\va_\vr)\|_{L^{\infty}}^{2^*}&\lesssim \int_{M\setminus B_{2\de}(p_0)}|T(\bar\va_\vr)|_\ig^{2^*}d\vol_\ig  \lesssim\int_M\Big( |\bar\va_\vr|_\ig^{2^*-1}|T(\bar\va_\vr)|_{\ig}+|T(\bar\va_\vr)|_{\ig}^{2^*-1}|\bar\va_\vr|_\ig   \Big)d\vol_\ig \\
&\leq \|T(\bar\va_\vr)\|_{L^{\infty}}\int_M|\bar\va_\vr|_\ig^{2^*-1} d\vol_\ig + \|T(\bar\va_\vr)\|_{L^{\infty}}^{2^*-1}\int_M|\bar\va_\vr|_\ig d\vol_\ig.
\endaligned
\]
Then, using the facts $2^*>2$ and $\|T(\bar\va_\vr)\|_{L^{\infty}}\to0$ as $\vr\to0$, we conclude from \eqref{l1-norm} that
\begin{\equ}\label{l-infty-norm}
	%\rx{\|T(\bar\va_\vr)\|_{L^1} \lesssim \|\bar\va_\vr\|_{L^1}} \lesssim \vr^{\frac{m-1}2}
	%\qquad\text{and}\qquad
	 \|T(\bar\va_\vr)\|_{L^\infty} \lesssim %\|T(\bar\va_\vr)\|_{L^1} \lesssim
	 \vr^{\frac{(m-1)^2}{2(m+1)}}.
\end{\equ}
And thus we get
\begin{eqnarray}\label{test-T-1}
		&&\Big| \int_M|\bar\va_\vr-T(\bar\va_\vr)|_\ig^{2^*}-|\bar\va_\vr|_\ig^{2^*} d\vol_\ig \Big|  \nonumber \\
		&&\qquad 
		\leq \Big|\int_M|\bar\va_\vr-\theta T(\bar\va_\vr)|_\ig^{2^*-2}(\bar\va_\vr-\theta T(\bar\va_\vr), T(\bar\va_\vr))_\ig \,d\vol_\ig \Big| \nonumber\\
		&&\qquad
		\leq \int_M|\bar\va_\vr-\theta T(\bar\va_\vr)|_\ig^{2^*-1}\cdot|T(\bar\va_\vr)|_\ig\,d\vol_\ig  \nonumber \\
		&&\qquad 
		\lesssim \int_M \Big( |\bar\va_\vr|_\ig^{2^*-1}|T(\bar\va_\vr)|_\ig+|T(\bar\va_\vr)|_\ig^{2^*} \Big) d\vol_\ig  \nonumber \\
		&&\qquad 
		\lesssim \|T(\bar\va_\vr)\|_{L^\infty}  \int_M|\bar\va_\vr|_\ig^{2^*-1}d\vol_\ig + \|T(\bar\va_\vr)\|_{L^\infty}^{2^*}
		\lesssim \vr^{\frac{m(m-1)}{m+1}}
\end{eqnarray}
where $\theta\in(0,1)$. Moreover, for any $\psi\in E$ with $\|\psi\|\leq1$, we can deduce from \eqref{l1-norm} and \eqref{l-infty-norm} that
\begin{eqnarray}\label{test-T-2}
	&&\Big| \int_M|\bar\va_\vr-T(\bar\va_\vr)|_\ig^{2^*-2}
	(\bar\va_\vr-T(\bar\va_\vr),\psi)_\ig -
	|\bar\va_\vr|_\ig^{2^*-2}(\bar\va_\vr,\psi)_\ig d\vol_\ig \Big|
	\nonumber\\
	&& \qquad \qquad\leq
	(2^*-1)\int_M|\bar\va_\vr-\theta
	T(\bar\va_\vr)|_\ig^{2^*-2}\cdot
	|T(\bar\va_\vr)|_\ig\cdot |\psi|_\ig d\vol_\ig  \nonumber \\
	&& \qquad \qquad\lesssim
	\int_M\Big(|\bar\va_\vr|_\ig^{2^*-2}\cdot|T(\bar\va_\vr)|_\ig\cdot
	|\psi|_\ig+ |T(\bar\va_\vr)|_\ig^{2^*-1}\cdot|\psi|_\ig \Big) d\vol_\ig
	\nonumber  \\
	&& \qquad \qquad \lesssim
	\|T(\bar\va_\vr)\|_{L^\infty}  \cdot \Big(\int_M|\bar\va_\vr|^{\frac{4m}{m^2-1}}d\vol_\ig\Big)^{\frac{m+1}{2m}}
	+ \|T(\bar\va_\vr)\|_{L^\infty}^{2^*-1}   %\nonumber\\
%	&& \qquad \qquad \lesssim \ \begin{cases}
%		\vr^{\frac23} + \vr^{\frac12} &\text{if } m=2, \\
%		\vr^{\frac32}|\ln\vr|^{\frac23}  + \vr&\text{if } m=3,\\
%		\vr^{\frac{(m-1)^2}{2(m+1)}+1} + \vr^{\frac{m-1}2}&\text{if } m\geq4
%	\end{cases} \nonumber \\
	\lesssim \vr^{\frac{m-1}2}. 
\end{eqnarray}
Notice that $\widetilde\Phi(\bar\va_\vr-P^0\bar\va_\vr)=\widetilde\Phi(\bar\va_\vr)$ and  $\nabla\widetilde\Phi(\bar\va_\vr-P^0\bar\va_\vr)=\nabla \widetilde\Phi(\bar\va_\vr)$ (this is because $T(\bar\va_\vr-P^0\bar\va_\vr)= T(\bar\va_\vr)-P^0\bar\va_\vr$ from the definition of $T$). We can then infer from \eqref{test-T-1}-\eqref{test-T-2} and Lemmas \ref{derivative-esti} and \ref{energy-esti} that $\{\widetilde\va_\vr:=\bar\va_\vr-P^0\bar\va_\vr\}_{\vr>0} \subset E^+\op E^-$ contains a $(PS)$-sequence for $\widetilde\Phi$ (by considering a decreasing sequence $\vr_n\searrow 0$). Combining these facts and  Theorem \ref{abstract thm}, we finally obtain
\[
\aligned
\tilde c&=\inf_{\va\in E^+\setminus\{0\}}\max_{\psi\in\R^+\va\op E^-}\widetilde\Phi(\psi)\\
&\leq\frac1{2m}\left(\frac{m}2\right)^m\om_m- C|W_\ig(p_0)|^2\vr^4 + \begin{cases}
	O(\vr^{\frac{30}7}) &\text{if } m=6,\\
	O(\vr^5) &\text{if } m\geq7,
\end{cases}
\endaligned
\]
which suggests $\tilde c< \ga_{crit}$. This, together with Lemma \ref{PS Dirac}, implies that $\tilde c$ is a critical value of $\widetilde\Phi$ (and hence of $\Phi$) and is attained.

\end{document}